%% file: main.tex
\documentclass[a4paper, oneside, british, 11pt]{article}

\usepackage{geometry} % page geometry
\usepackage{fullpage} % smaller margins
\usepackage{fontspec} % displays special symbols
\usepackage{babel} % language related package
\usepackage{float} % better figures
\usepackage{authblk} % author's affiliation
\usepackage{lipsum} % Lorem Ipsum
\usepackage{csquotes} % quotations
\usepackage{wrapfig} % wraps text around small figures
\usepackage{datetime} % dates
\usepackage{subcaption} % subfigures and other subcaptions
\usepackage{multicol} % multiple columns of text
\usepackage[bottom]{footmisc} % footnotes at bottom of page
\usepackage[disable]{todonotes} % to do list
\usepackage{hyperref} % clickable links

\usepackage{enumitem} % better enumerate
\setlist[enumerate]{nosep, label=(\roman*), leftmargin=*}

\newcommand{\simnote}[1]{\todo[inline, color=green!20]{{\it Simon:~}#1}}
\newcommand{\highlight}{\color{blue}}
\renewcommand{\highlight}{}

\usepackage[maxbibnames=999]{biblatex} % bibliography 
\usepackage{graphicx} % images
\graphicspath{{images/}}

\usepackage{amsmath, amssymb, amsfonts}
\usepackage{mathtools}
\usepackage{tikz-cd}
    \usetikzlibrary{arrows,arrows.meta}
    \tikzcdset{arrow style=tikz, diagrams={>=stealth'}}
\usepackage{ebproof}
\usepackage{stmaryrd}

\usepackage{amsthm}

\theoremstyle{plain}
\newtheorem{thm}{Theorem}[section]

\newtheorem{claim}[thm]{Claim}

\newtheorem{prop}[thm]{Proposition}
\newtheorem{lemma}[thm]{Lemma}
\newtheorem{coro}[thm]{Corollary}

\theoremstyle{definition}
\newtheorem{definition}[thm]{Definition}

\theoremstyle{remark}

\newtheorem{rmk}[thm]{Remark}

\renewcommand{\epsilon}{\varepsilon}

\newcommand{\mrm}{\mathrm}

\newcommand{\mcal}{\mathcal}

\newcommand{\mrel}{\mathrel}

\newcommand{\msf}{\mathsf}

\makeatletter
\title{Superamalgamation for {\highlight non-distributive} modal lattices via dualities}
\author[1]{Rodrigo Nicolau Almeida}
\author[1]{Nick Bezhanishvili}
\author[1, 2]{Simon Lemal}
\affil[1]{Institute for Logic, Language and Computation, Universiteit van Amsterdam}
\affil[2]{Department of Mathematics, Université du Luxembourg}
\date{}
\makeatother

\begin{document}

\maketitle

\begin{abstract}
   We show that the variety of modal lattices has the superamalgamation property. As a consequence, we obtain that the weak positive modal logic has the Craig interpolation property. Our proof employs the recent duality for modal lattices based on modal L-spaces. Moreover, we extend this result to a number of other weak positive modal logics axiomatized by modal axioms corresponding to universal Horn sentences.
\end{abstract}

\input{body}

\printbibliography

\end{document}

%% file: body.tex
\section{Introduction}

{\highlight Superamalgamation is an important property in algebraic logic. In particular, it often  provides an algebraic counterpart of the notion of Craig interpolation.} More formally, the standard connection is as follows: a logic $\mcal{L}$ has the Craig interpolation property if and only if its corresponding variety of algebras $\msf{K}(\mcal{L})$ has the superamalgamation property.
We refer to the recent handbook on interpolation \cite{craig2026} for the current state of the art on Craig interpolation, as well as to Chapters \cite{bez2026, fus2026, met2026} of this handbook for detailed discussions of the connections between (super)amalgamation and Craig interpolation.

Maksimova's classical result \cite{mak1977} states that there are exactly seven consistent superintuitionistic logics with Craig interpolation. This result was proved by characterizing all seven nontrivial varieties of Heyting algebras with the amalgamation property (in varieties of Heyting algebras, the notions of amalgamation and superamalgamation coincide); for the proof we refer, for example, to \cite[Section 14.4]{cz1997}. Using similar techniques, Maksimova \cite{mak1979} also showed that there are at most 37 modal logics extending $\mathsf{S4}$ that have the Craig interpolation property (see, e.g., \cite{cz1997, fus2026, gab2005}). {\highlight It was shown (by Maksimova, Boolos, Rautenberg, Shimura) that 31 of these do have the Craig interpolation property. This characterization was recently finalized by Santchi and Vooijs \cite{san2026} who, in fact, proved that the remaining six logics also have the property and thus all these 37 logics have the Craig interpolation property (see \cite{san2026} for references and further history of these results).}

In this paper, we study a natural version of the Craig interpolation property for modal logics lacking an implication connective which are moreover not necessarily distributive.  Such logics were named \emph{weak positive modal logics} in \cite{bez2024}, where ``weak'' refers to the possible lack of the distributivity law. Most such results are proved by establishing that the class of modal lattices has \emph{superamalgamation}. In this context, Jónsson \cite{jon1956} (see also \cite{jip1992}) proved that the variety of lattices has the amalgamation property. Subsequently, \cite{day1984} showed that the only nontrivial varieties of lattices with the amalgamation property are the variety of all lattices and the variety of distributive lattices. {\highlight Note also that the variety of distributive lattices lacks the superamalgamation property and also the strong amalgamation property (see e.g. \cite{fri1990}, \cite[Ch.~2]{hoo2001}), making the possible lack of distributivity a natural assumption.}

{\highlight The proof of Jónnson's result presented in \cite{jip1992} is algebraic,} making use of filters on lattices; whilst this allows for ample generalization, when considering modalities, {\highlight is is not obvious whether these modalities could be defined on the filters. To overcome this,} we approach this problem through duality theory. Duality is a powerful tool that provides topological representations of algebraic structures. For dualities for Boolean algebras, distributive lattices, Heyting algebras, and modal algebras we refer to \cite{dav2002, esa2019, geh2024, ven2007}. In all these cases, duality yields a transparent dual proof of the amalgamation property: briefly, one dualizes all objects and morphisms and then constructs an amalgam using the dual construction of a pullback (in an appropriate category). For proofs in the case of modal algebras, we refer to \cite{bez2026, ven2007}.

In the setting of modal lattices that are not necessarily distributive,  a new duality, resembling the classical dualities in the distributive setting, was recently developed in \cite{bez2024}. We briefly recall the main points of this duality. Given a lattice $L$, one considers the set of all filters of $L$, equipped with the patch topology of the standard Stone topology, ordered by inclusion of filters. This order is easily seen to form a meet semilattice, the topology is compact, and whenever $x\not\preceq y$, then $x$ and $y$ are separated by a clopen filter---this yields the structure of an \emph{L-space}. Given an L-space $(X, \tau)$, one obtains a lattice by taking the lattice of all clopen filters of $X$. In this way, every lattice is represented as the lattice of clopen filters of its dual L-space. This correspondence can be lifted to a representation and  duality of modal lattices \cite{bez2024}.

Using this duality, and essentially the same ideas as those employed in the distributive setting, we show that the variety of modal lattices has the superamalgamation property. Consequently, the corresponding weak positive modal  logic has the Craig interpolation property. There are both similarities and differences of this approach with the classical case. In the classical setting, one works with the complex algebra of a frame (i.e., the powerset algebra or the algebra of all upsets in the distributive and Heyting cases); these correspond to canonical extensions of modal algebras, distributive lattices or Heyting algebras 
\cite{esa2019, geh2024, ven2007}. The analogue in our setting is the modal lattice of all filters of a meet semilattice, which in turn corresponds to the $\Pi_1$-completion of modal lattices \cite{bez2024}. 

Moreover, as in the classical case, our proof extends to all logics whose frame correspondents are axiomatized by universal Horn sentences \cite{bez2026, ven2007}. We give several illustrative examples in Section~\ref{section:supamal_modal}. Surprisingly, the weak positive modal logic axiomatized by the Church--Rosser $\msf{.2}$ axiom $\Diamond\Box p\trianglelefteq\Box\Diamond p$ also has the Craig interpolation property (see Theorem~\ref{thm:supamal_ext}). This should be contrasted with the classical case, where the modal logic $\msf{K.2}$ lacks the Craig interpolation property \cite[Section 5.6.2]{marx1995}. Moreover, we show that despite lacking superamalgamation, the logic of distributive lattices enjoys the Craig interpolation property.

The paper is organized as follows. In Section 2, we first recall the syntax and algebraic semantics of weak positive modal logic and then introduce the algebraic counterparts of interpolation related properties. In Section 3, we present a new proof of the superamalgamation property for lattices, first recalling the relevant duality and then using it to establish our result. In Section 4, we build on the results of Section 3 by adding modalities. This is followed by a discussion of the Craig interpolation property for certain extensions of weak positive modal logic. In the final section, we offer some concluding remarks.

\section{Preliminaries}

\subsection{Modal lattices and weak positive modal logic}

We assume the reader is familiar with basic concepts from lattice theory \cite{gra2011} and universal algebra \cite{burris_sanka}. {\highlight We assume all lattices in this paper to be bounded. The \emph{language} of lattices is given by the BNF:
\begin{equation*}
    p \mid \top \mid \bot \mid \phi\land\psi \mid \phi \lor \psi.
\end{equation*}

The \emph{language} of modal lattices is obtained by extending it to the following BNF:}
\begin{equation*}
    p \mid \top \mid \bot \mid \phi\land\psi \mid \phi \lor \psi \mid \Box \phi \mid \Diamond\phi.
\end{equation*}

{\highlight Throughout this section, most statement apply to both lattices and modal lattices. In that case, we refer to (modal) lattices.}

\begin{definition}
    Given a set of propositional letters $P$, we let $Fm_P$ be the set of formulas generated from those propositional letters using the language of (modal) lattices.

    A \emph{consequence pair} is an expression of the form $\phi\trianglelefteq\psi$, with $\phi, \psi\in Fm_P$.
\end{definition}

\begin{definition}
    The deduction rules and axioms for consequence pairs are as follows:

    \begin{center}
    \renewcommand{\arraystretch}{2}
    \begin{tabular}{ccr}
        $p\trianglelefteq\top$, & $\bot\trianglelefteq p$, & \emph{top} and \emph{bottom} \\
        $p\trianglelefteq p$, & \begin{prooftree}
            \hypo{p\trianglelefteq q}\hypo{q\trianglelefteq r}\infer2{p\trianglelefteq r}
        \end{prooftree}, & \emph{reflexivity} and \emph{transitivity} \\
        $p\land q\trianglelefteq p$, & $p\land q\trianglelefteq q$, & \emph{left conjunction rules} \\
        $p\trianglelefteq p\lor q$, & $q\trianglelefteq p\lor q$, & \emph{right disjunction rules} \\
        \begin{prooftree}
            \hypo{p\trianglelefteq q}\hypo{p\trianglelefteq r}\infer2{p\trianglelefteq q\land r}
        \end{prooftree}, & \begin{prooftree}
            \hypo{p\trianglelefteq r}\hypo{q\trianglelefteq r}\infer2{p\lor q\trianglelefteq r}
        \end{prooftree}, & \emph{right conjunction} and \emph{left disjunction} \\
        $\top\trianglelefteq\Box\top$, & $\top\trianglelefteq\Diamond\top$, & \emph{modal top} \\
        \begin{prooftree}
            \hypo{p\trianglelefteq q}\infer1{\Box p\trianglelefteq\Box q}
        \end{prooftree}, & \begin{prooftree}
            \hypo{p\trianglelefteq q}\infer1{\Diamond p\trianglelefteq\Diamond q}
        \end{prooftree}, & \emph{Becker's rules} \\
        $\Box p\land\Box q\trianglelefteq\Box(p\land q)$, & $\Diamond p\land\Box q\trianglelefteq\Diamond(p\land q)$, & \emph{linearity} and \emph{duality}
    \end{tabular}
    \end{center}

    We denote by $\mcal{L}$ the smallest set of consequence pairs closed under uniform substitution and the axioms and rules listed above.

    If $\Gamma$ is a set of consequence pairs, $\mcal{L}(\Gamma)$ denotes the smallest set of modal consequence pairs containing $\Gamma$, {\highlight closed under uniform substitution} and the axioms and rules listed above. Such a set is called a \emph{weak positive modal logic}, or a \emph{logic} for short. We write $\phi\vdash_\Gamma\psi$ if $\phi\trianglelefteq\psi\in\mcal{L}(\Gamma)$.
\end{definition}

\begin{definition}\label{def:mod_lat}
    A \emph{modal lattice} is a tuple $(A, \top, \bot, \land, \lor, \Box, \Diamond)$ such that $(A, \top, \bot, \land, \lor)$ is a bounded lattice, and such that the following identities hold:
    \begin{alignat*}{4}
        \top&\approx\Box\top, & \top&\approx\Diamond\top, \\
        \Box a\land\Box b&\approx\Box(a\land b), \qquad&\qquad \Diamond a\lor\Diamond b&\le \Diamond(a\lor b), \\
        \Diamond a\land\Box b&\le \Diamond(a\land b).
    \end{alignat*}

    We denote the category of modal lattices by $\msf{MLat}$, and the category of lattices by $\msf{Lat}$.
\end{definition}

\begin{rmk}
   {\highlight Modal lattices, as defined in Definition~\ref{def:mod_lat}, were introduced and studied in \cite{bez2024}. Modal operators on bounded lattices---though without additional connections between the $\Box$ and $\Diamond$ modalities---had been studied earlier, for example, by Gehrke and Harding \cite{geh2001} and Jipsen and Moshier \cite{mos2014}. By contrast, modal operators on distributive lattices have a rich history, dating back at least to the work of Goldblatt \cite{gol1989}, Dunn \cite{dunn1995}, Jansana and Celani \cite{cel1999}, and Sofronie-Stokkermans \cite{sof2000a, sof2000b}, among others.  In \cite{cas2004} the term ``modal lattices'' was used to denote distributive lattices with  $\Box$ and  $\Diamond$ operators satisfying the normality axioms, but no additional equations.}  
\end{rmk}

\begin{rmk}

    We note that the dual Dunn axiom, $\Box(a\lor b)\le \Box a\lor \Diamond b$ {\highlight does not follow from the axioms of modal lattices} \cite[Remark 4.8]{bez2024}, and similarly, the symmetric normality axiom for $\Diamond$, namely $\Diamond(a\lor b)\le \Diamond a\lor \Diamond b$, {\highlight also does not follow from those axioms} \cite[Example 4.42]{bez2024}.

    {\highlight The choice of the above axiomatization is determined by the intended Kripke-like semantics (see Section~\ref{subsec:dual} and \cite{bez2024}). In the intended semantics, all the axioms of modal lattices hold, whereas the aforementioned additional axioms fail (see again \cite[Remark 4.8]{bez2024}, and \cite[Example 4.42]{bez2024}). Thus, although the dual Dunn axioms and the normality axiom for $\Diamond$ can be added to the axioms of modal lattices, they are not forced by the intended Kripke-like semantics. It is therefore more general not to assume them.

    We also point out that the asymmetry between $\Box$ and $\Diamond$ is quite common when moving beyond classical modal logic. In fact, these modalities behave in almost the same way as their counterparts in Constructive Modal Logic \cite{bel2001, pai2003, ale2001}. Modal lattices have one additional axiom, namely $\Diamond \top = \top$, which is also forced by the semantics. Constructive Modal Logics together with this axiom have also been studied in the literature  \cite{wij1990, wij2005}.}
\end{rmk}

We can use (modal) lattices as models of weak positive (modal) logic: the set $Fm_{P}$ comes naturally equipped {\highlight with a structure in the language of (modal) lattices}, and so given a (modal) lattice $A$, we call a homomorphism $\sigma\colon Fm_{P}\to A$ a \emph{valuation}.

\begin{definition}
    We say that a (modal) lattice $A$ \emph{validates} a consequence pair $\phi\trianglelefteq\psi$, written $A\vDash\phi\trianglelefteq\psi$, if for every valuation $\sigma\colon Fm_P\to A$, we have $\sigma(\phi)\le\sigma(\psi)$.
    
    A variety $\msf{K}$ \emph{validates} a consequence pair $\phi\trianglelefteq\psi$, written $\vDash_{\msf{K}}\phi\trianglelefteq\psi$, if for all $A\in\msf{K}$, we have $A\vDash\phi\trianglelefteq\psi$.
\end{definition}

\begin{definition}
    Given a set of consequence pairs $\Gamma$, we denote by $\msf{K}(\Gamma)$ the variety of (modal) lattices that validate all the pairs in $\Gamma$. We use $\vDash_\Gamma$ as a shorthand for $\vDash_{\msf{K}(\Gamma)}$.
    
    Given a variety $\msf{K}$, we denote by $\mcal{L}(\msf{K})$ the set of consequence pairs that are validated by $\msf{K}$. This gives a correspondence between varieties of (modal) lattices and (modal) logics.
\end{definition}

\begin{definition}
    Given a variety $\msf{K}$, we denote by $F_\msf{K}(P)$ the free algebra on the set of generators $P$ in $\msf{K}$. When the variety is of the form $\msf{K}(\Gamma)$, we denote the free algebra by $F_\Gamma(P)$. It is obtained from $Fm_P$ by identifying pairs $\phi, \psi$ such that $\phi\vdash_\Gamma\psi$ and $\psi\vdash_\Gamma\phi$.
\end{definition}

\begin{rmk}
    For all of the objects defined above, we drop the subscript $\Gamma$ when $\Gamma$ is the empty set.
\end{rmk}

Free algebras play a key role in the algebraic completeness theorems, and they will be used heavily throughout our work. We collect some basic facts concerning them.

\begin{prop}\label{prop:free_alg}\
\begin{enumerate}
        \item\label{item:prop:free_alg:i} Given $\phi\in Fm_{P}$, and $\mathsf{K}$ a variety of (modal) lattices, we can identify $\phi$ with $[\phi]_{\mathsf{K}}$, its equivalence class in the free algebra $F_{\mathsf{K}}(P)$. Under this identification, we have that {\highlight $\vDash_{\mathsf{K}}\phi\trianglelefteq\psi$} if and only if $\phi\le\psi$ holds in $F_{\mathsf{K}}(P)$.
        %\item\label{item:prop:free_alg:ii} Given a modal lattice $A$, homomorphisms $\tilde{\sigma}\colon F_{\mathsf{K}}(\overline{p})\to A$ are all induced by assignments $\sigma\colon \overline{p}\to A$. We call this an \emph{extension} or a \emph{lift} of $\sigma$.
        \item\label{item:prop:free_alg:iii} If $\sigma_{1}\colon F_{\mathsf{K}}(\overline{p},\overline{q})\to A$ and $\sigma_{2}\colon F_{\mathsf{K}}(\overline{p},\overline{r})\to A$ are homomorphisms which agree on $\overline{p}$, then there is a unique morphism $\sigma\colon F_{\mathsf{K}}(\overline{p},\overline{q},\overline{r})\to A$ which agrees with $\sigma_{1}$ and $\sigma_{2}$ on the restrictions.
\end{enumerate}
\end{prop}

The following theorem is proved in \cite[Theorem 4.13]{bez2024}.

\begin{thm}[Algebraic Completeness Theorem]
    For $\phi,\psi\in Fm_{P}$ and $\Gamma$ a set of formulas, we have $\phi\vdash_\Gamma\psi$ iff $\vDash_\Gamma\phi\trianglelefteq\psi$.
\end{thm}

\subsection{Craig interpolation and Superamalgamation}

We move on to studying Craig interpolation. Usually Craig interpolation is defined for languages containing an implication connective. The following {\highlight definition adapts it to our setting}, treating the $\vdash$ as a form of implication:

\begin{definition}
    We say that a logic $\mcal{L}(\Gamma)$ has the \emph{Craig interpolation property} if for all $\phi\in Fm_{\overline{p}, \overline{q}}$ and $\psi\in Fm_{\overline{p}, \overline{r}}$ such that $\phi\vdash_\Gamma\psi$, there is a formula $\chi\in Fm_{\overline{p}}$ such that $\phi\vdash_\Gamma\chi$ and $\chi\vdash_\Gamma\psi$. Such a formula $\chi$ is called an \emph{interpolant}.
\end{definition}

\begin{definition}
    Let $\msf{K}$ be a variety of (modal) lattices. A \emph{V-formation} is a tuple of lattices $K, L_1, L_2\in\msf{K}$ and embeddings $h_i\colon K\to L_i$, as in Figure~\ref{subfig:vform}

    A variety $\msf{K}$ of (modal) lattices has the \emph{amalgamation property} if for every V-formation, there is a lattice $M$ {\highlight in $\msf{K}$}, and {\highlight embeddings} $p_i\colon L_i\to M$ such that $p_1\circ h_1 = p_2\circ h_2$, or in other words, such that the diagram in Figure~\ref{subfig:amal} commutes. The lattice $M$ is called an \emph{amalgam}.

    \begin{figure}[ht]
        \centering
        \begin{subfigure}{.49\textwidth}
            \centering
            \begin{tikzcd}
                & \phantom{a} & \\
                L_1 & & L_2 \\
                & K \arrow[lu, "h_1", hook] \arrow[ru, "h_2"', hook] &                             
            \end{tikzcd}
            \caption{V-formation}
            \label{subfig:vform}
        \end{subfigure}
        \hfill
        \begin{subfigure}{.49\textwidth}
        \centering
            \begin{tikzcd}
                & M & \\
                L_1 \arrow[ru, "p_1", hook] & & L_2 \arrow[lu, "p_2"', hook] \\
                & K \arrow[lu, "h_1", hook] \arrow[ru, "h_2"', hook] &                             
            \end{tikzcd}
            \caption{Amalgam}
            \label{subfig:amal}
        \end{subfigure}
        \caption{Amalgamation of lattices}
    \end{figure}
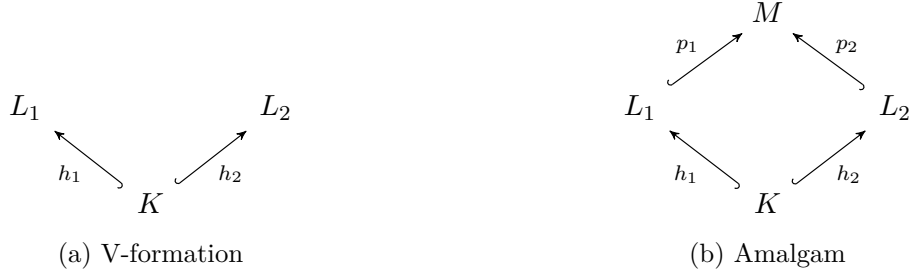

    A variety $\msf{K}$ of (modal) lattices has the \emph{superamalgamation property} if for every V-formation, there is an amalgam such that $p_1(a)\le p_2(b)$ for $a\in L_1$ and $b\in L_2$ implies the existence of $c\in K$ such that $a\le h_1(c)$ and $h_2(c)\le b$. The lattice $M$ is called a \emph{superamalgam}.
\end{definition}

\begin{definition}
    A variety $\msf{K}$ of (modal) lattices has \emph{epimorphism surjectivity} if every epimorphism in the category corresponding to the variety is onto.
\end{definition}

\begin{rmk}
    The epimorphism surjectivity property is the algebraic equivalent of the Beth definability property.
    However, we decide to not go into the details of the Beth definability property here (see \cite{hoo2001} for details).
\end{rmk}

\begin{lemma}\label{lemma:free_amal}
    Let $\msf{K}$ be a variety of lattices {\highlight and $\overline{p}, \overline{q}, \overline{r}\subseteq P$}. The following are equivalent:
    \begin{enumerate}
        \item\label{item:lemma:free_amal:i} the V-formation in Figure~\ref{subfig:free_vform} below superamalgamates;
        \item\label{item:lemma:free_amal:ii} the amalgamation diagram in Figure~\ref{subfig:free_amal} below is a superamalgamation;
        \item\label{item:lemma:free_amal:iii} the logic $\mcal{L}(\msf{K})$ has the {\highlight property that for any $\phi\in Fm_{\overline{p}, \overline{q}}$ and $\psi\in Fm_{\overline{p}, \overline{r}}$ such that $\vDash_\msf{K}\phi\trianglelefteq\psi$, there is $\chi\in Fm_{\overline{p}}$ such that $\vDash_\msf{K}\phi\trianglelefteq\chi$ and $\vDash_\msf{K}\chi\trianglelefteq\psi$.}
    \end{enumerate}
\end{lemma}

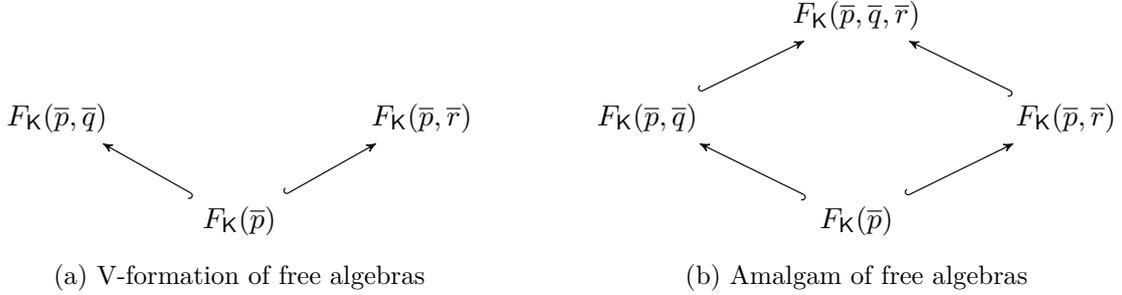
\begin{figure}[ht]
    \centering
    \begin{subfigure}{.49\textwidth}
        \centering
        \begin{tikzcd}
            & \phantom{a} & \\
            F_\msf{K}(\overline{p}, \overline{q}) & & F_\msf{K}(\overline{p}, \overline{r}) \\
            & F_\msf{K}(\overline{p}) \arrow[lu, hook] \arrow[ru, hook] &        
        \end{tikzcd}
        \caption{V-formation of free algebras}
        \label{subfig:free_vform}
    \end{subfigure}
    \hfill
    \begin{subfigure}{.49\textwidth}
    \centering
        \begin{tikzcd}
            & F_\msf{K}(\overline{p}, \overline{q}, \overline{r}) & \\
                F_\msf{K}(\overline{p}, \overline{q}) \arrow[ru, hook] & & F_\msf{K}(\overline{p}, \overline{r}) \arrow[lu, hook] \\
            & F_\msf{K}(\overline{p}) \arrow[lu, hook] \arrow[ru, hook] &                         
        \end{tikzcd}
        \caption{Amalgam of free algebras}
        \label{subfig:free_amal}
    \end{subfigure}
    \caption{Craig interpolation and superamalgamation on free algebras}
\end{figure}

\begin{proof}
    That \ref{item:lemma:free_amal:iii} implies \ref{item:lemma:free_amal:ii} follows simply by unfolding the definitions. That \ref{item:lemma:free_amal:ii} implies \ref{item:lemma:free_amal:i} is trivial. Let us prove that \ref{item:lemma:free_amal:i} implies \ref{item:lemma:free_amal:iii}. Assume that we have a superamalgam $M$ with two monos $\sigma_1\colon F_\msf{K}(\overline{p}, \overline{q})\to M$ and $\sigma_{2}\colon F_\msf{K}(\overline{p}, \overline{r})\to M$ which agree on $F_\msf{K}(\overline{p})$. Assume that $\phi\in F_\msf{K}(\overline{p}, \overline{q})$ and $\psi\in F_\msf{K}(\overline{p}, \overline{r})$ are such that {\highlight $\vDash_\msf{K}\phi\trianglelefteq\psi$}. Then by  Proposition~\ref{prop:free_alg}\ref{item:prop:free_alg:iii}, there is a map $\sigma\colon F_\msf{K}(\overline{p}, \overline{q}, \overline{r})\to M$ extending $\sigma_{1}, \sigma_{2}$. As {\highlight $\vDash_\msf{K}\phi\trianglelefteq\psi$}, we have $\sigma_{1}(\phi) = \sigma(\phi)\le\sigma(\psi) = \sigma_{2}(\psi)$. Therefore by superamalgamation, there is $\chi\in F_\msf{K}(\overline{p})$ such that $\phi\le\chi$ in $F_\msf{K}(\overline{p}, \overline{q})$, and $\chi\le\psi$ in $F_\msf{K}(\overline{p}, \overline{r})$. By Proposition~\ref{prop:free_alg}\ref{item:prop:free_alg:i}, this implies that {\highlight $\vDash_\msf{K}\phi\trianglelefteq\chi$} and {\highlight $\vDash_\msf{K}\chi\trianglelefteq\psi$}, thus $\chi$ is the interpolant for $\phi$ and $\psi$.
\end{proof}

Recall that a bounded lattice $L$ is called a \emph{Heyting algebra} if for all $a,b\in L$ there is an element $a\rightarrow b$ such that for all $c\in L$:
\[ a\land c\le b \qquad\text{iff}\qquad c\le a\rightarrow b. \]
For basic facts about Heyting algebras see \cite[Chapter 7]{cz1997}.

Recall that \emph{positive logic} is the logic of distributive lattices, {\highlight as well as the $\neg$-free fragment of classical logic, see e.g. \cite{cel1999} for more details. Using the above ideas we obtain a simple proof of Craig interpolation for positive logic.}

\begin{prop}\label{prop:pos_int}
    Positive logic has Craig interpolation.
\end{prop}
\begin{proof}
    We denote the variety of distributive lattices by $\msf{DLat}$.
    By Lemma~\ref{lemma:free_amal}, it is sufficient to show that {\highlight any V-formation of free distributive lattices superamalgamates in the class of distributive lattices}. However, any free distributive lattice is a Heyting algebra, and the inclusion of $\iota\colon F_\msf{DLat}(\{p_{1},\dots,p_{n}\})\to F_\msf{DLat}(\{p_{1},\dots,p_{n},p_{n+1}\})$ is a Heyting algebra homomorphism  \footnote{For a proof and references to this result, see \cite[Prop. 2.12]{alm2024}.}, and the variety of Heyting algebras has superamalgamation \cite{mak1977}.
\end{proof}

\begin{rmk}\highlight 
    It is well known that classical logic enjoys the stronger \emph{Lyndon interpolation property} \cite{lyn1959}: whenever $\phi(\overline{p},\overline{q})\rightarrow \psi(\overline{q},\overline{r})$, then one can find an interpolant such that every instance of a variable which occurs only positively in both {the antecedent and the consequent} occurs only positively in the interpolant, and same for negative variables. Since in positive logic, every variable certainly occurs only positively, this immediately entails that positive logic has Craig interpolation. As far as we are aware, the proof of Proposition \ref{prop:pos_int} is new, and we hope it might have independent interest.
\end{rmk}

The following {\highlight classical result gives} (see e.g. \cite[pp.~93]{kiss1983}) us a natural connection between superamalgamation, Craig interpolation and epimorphism surjectivity\footnote{Note that Item \ref{item:thm:inter_amal:i} also easily follows from Lemma~\ref{lemma:free_amal}.}:

\begin{thm}\label{thm:inter_amal}
    Let $\msf{K}$ be a variety of (modal) lattices and $\mcal{L}(\msf{K})$ be the corresponding logic.
    \begin{enumerate}
        \item\label{item:thm:inter_amal:i} If $\msf{K}$ has superamalgamation, then $\mcal{L}(\msf{K})$ has Craig interpolation.

        \item\label{item:thm:inter_amal:ii} If $\msf{K}$ has superamalgamation, then it has epimorphism surjectivity.
    \end{enumerate}
\end{thm}

% \begin{proof}
%     \ref{item:thm:inter_amal:i} It follows from Lemma~\ref{lemma:free_amal} that $\mcal{L}(\msf{K})$ has Craig interpolation.

%     \ref{item:thm:inter_amal:ii} This is given in \cite[Corollary 2.5.23]{hoo2001}. 
% \end{proof}

\begin{rmk}
    The converse to \ref{item:thm:inter_amal:i} is true when $\msf{K}$ is a variety of Heyting algebras. Note that {\highlight in general, Craig interpolation need not imply superamalgamation. For example,} Corollary~\ref{prop:pos_int} tells us that the logic of distributive lattices has Craig interpolation. However, distributive lattices do not have epimorphism surjectivity, and therefore, fail to have superamalgamation. Indeed, the inclusion from the $3$-chain to the $4$ elements Boolean algebra (see Figure~\ref{fig:fail_epi_surj}) depicted below is an epimorphism, but it is not surjective. 
    
    {\highlight There is also another way to observe that distributive latices lack superamalgamation. Distributive lattices  lack the strong amalgamation property \cite{fri1990}. It is also known that superamalgamation implies the strong amalgamation property \cite[Definitions~2.5.1 \& 2.5.10]{hoo2001}. Therefore, distributive lattice lack superamalgamation.}
    
    \begin{figure}[ht]
        \centering
        \begin{tikzpicture}
            \node at (0,0) {$\bullet$};
            \node at (-1,1) {$\bullet$};
            \node at (0,2) {$\bullet$};
            \node at (3,0) {$\bullet$};
            \node at (2,1) {$\bullet$};
            \node at (4,1) {$\bullet$};
            \node at (3,2) {$\bullet$};
        
            \draw (0,0) -- (-1,1) -- (0,2);
        
            \draw (3,0) -- (2,1) -- (3,2) -- (4,1) -- (3,0);
        
            \draw [dotted,->] (0.3,0) -- (2.7,0);
        
            \draw [dotted,->] (-0.7,1) -- (1.7,1);
        
            \draw[dotted,->] (0.3,2) -- (2.7,2);
        \end{tikzpicture}
        \caption{An epimorphism which is not surjective}
        \label{fig:fail_epi_surj}
    \end{figure}
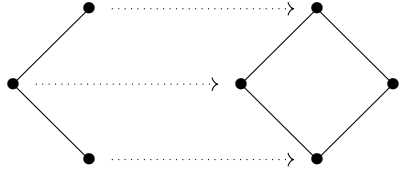    
\end{rmk}

\section{Superamalgamation for lattices}\label{sec:supamal_nonmodal}

In this section, we focus on (\emph{non modal}) lattices. {\highlight While the superamalgamation result for the varitey of lattices presented in this section is not novel, its proof is new, and serves as a convenient stepping stone for the results in the next section.}

\subsection{Duality for lattices}

We now recall the duality for lattices and modal lattices developed in \cite{bez2024}. Since it is obtained by restricting HMS (Hoffman--Mislove--Stralka) duality {\highlight for meet semilattice \cite{hms1974}, we first recall the main definitions of this duality. Throughout this paper, we assume all meet semilattices to have a top element.}

\begin{definition}
 {\highlight  An \emph{L-frame} is a tuple $(X, 1, \curlywedge)$, where $(X, \curlywedge)$ is a 
a meet semilattice with a top element $1$.}
We denote its underlying order by $\preceq$. Given an L-frame $X$\footnote{\highlight Here, and throughout this paper, we write $X$ instead of $(X, 1, \curlywedge)$.}, we denote by $\msf{Fil_F}(X)$ its set of filters, which is a complete lattice.
\end{definition}

\begin{definition}
    An \emph{HMS-space} is a tuple $(X, 1, \curlywedge, \tau)$ such that $(X, 1, \curlywedge)$ is an L-frame and $(X, \tau)$ is a compact topological space such that whenever $x\not\preceq y$, there is some clopen filter $U\subseteq X$ such that $x\in U$ and $y\not\in U$. An \emph{HMS-morphism} is a continuous meet semilattice homomorphism. We write $\msf{HMS}$ for the category of HMS-spaces and HMS-morphisms.
\end{definition}

\begin{thm}\label{thm:HMS_duality}
    The category $\msf{HMS}$ is dually equivalent to the category $\msf{MSL}$ of meet semilattices.
\end{thm}
\begin{proof}
    For a complete proof see \cite{hms1974}. We briefly recall the functors involved in this duality. One part is the assignment $\mathsf{Fil}\colon \mathsf{MSL}\to \mathsf{HMS}$, that sends a {\highlight meet semilattice} to its set of filters. Given a meet semilattice $L$, the set $\mathsf{Fil}(L)$ is a meet semilattice, with meet given by intersection. We equip it with a Stone like topology $\tau$ given by declaring the sets
    \begin{equation*}
        \varphi(a)\coloneq \{F\in \mathsf{Fil}(L) \mid a\in F\},
    \end{equation*}
    and its complements to be clopen. We obtain that $(\mathsf{Fil}(L),L,\cap,\tau)$ is an HMS-space. Conversely, if $(X, \tau)$ is an HMS-space, the set $\mathsf{ClopFil}(X)$ of clopen filters forms a meet semilattice. Given a meet semilattice homomorphism $h\colon K\to L$, the map $h^{-1}\colon \mathsf{Fil}(L)\to \mathsf{Fil}(K)$ will be a continuous meet semilattice map. And given an HMS-morphism $f\colon X\to Y$, the map $f^{-1}\colon \mathsf{ClopFil}(Y)\to \mathsf{ClopFil}(X)$ will be a meet semilattice homomorphism.
\end{proof}

In the setting of weak positive logic, the morphisms of interest between L-frames satisfy more requirements:

\begin{definition}
    An \emph{L-morphism} between L-frames $X$ and $Y$ is a semilattice homomorphism $f\colon X\to Y$ that satisfies for all $x\in X$ and $y', z'\in Y$:

    \begin{enumerate}
        \item if $f(x)=1$, then $x=1$;
        \item if $y'\curlywedge z'\preceq f(x)$, then there are $y, z\in X$ such that $y'\preceq f(y)$, $z'\preceq f(z)$ and $y\curlywedge z\preceq x$ (see Figure~\ref{fig:L-mor}).
    \end{enumerate}

    We denote the category of L-frames and L-morphisms by $\msf{LFrm}$.
\end{definition}

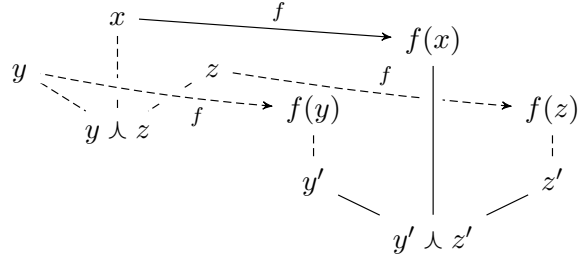
\begin{figure}[ht]
    \centering
    \begin{tikzcd}[scale=.95]
        & [-1.3em]
          x \arrow[rrrd, "f"]
            \arrow[dddd, dashed, -]
        & [-1.4em]
        & [-1em]
        & [-1.5em]
        & [-1.5em] \\ [-2.5em]
        &
        &
        &
        & f(x)
        & \\ [-2.2em]
      y     \arrow[ddr, dashed, -]
        &
        & z \arrow[ddl, dashed, -]
            \arrow[rrrd, bend left=-2, dashed, "f"]
        &&& \\ [-2.1em]
        &
        &
        & f(y) \arrow[dd, dashed, -]
               \arrow[lllu, bend left=3, dashed, <-, crossing over, "f" pos=.25]
        &
        & f(z) \arrow[dd, dashed, -] \\ [-2.6em]
        & y \curlywedge z
        &&&& \\ [-1.7em]
        &
        &
        & y' \arrow[dr, -]
        &
        & z' \arrow[dl, -] \\ [-1.5em]
        &
        &
        &
        & y' \curlywedge z'  \arrow[uuuuu, crossing over, -]
        &
    \end{tikzcd}
    \caption{The second L-morphism condition}
    \label{fig:L-mor}
\end{figure}

The definition of L-morphism ensures that whenever $f\colon X\to Y$ is an L-morphism, then $f^{-1}\colon\msf{Fil_F}(Y)\to\msf{Fil_F}(X)$ is a lattice homomorphism. This gives rise to a {\highlight contravariant} functor $\msf{Fil_F}\colon\msf{LFrm}\to\msf{Lat}$.

HMS duality can be restricted to lattices:

\begin{definition}
    An \emph{L-space} is an HMS-space $(X, 1, \curlywedge, \tau)$ such that $\{1\}$ is clopen and \[ U\lor V\coloneq \{x\in X \mid \exists y\in U, z\in V~\textrm{such that}~y \curlywedge z\preceq x\}, \] the filter generated by $U\cup V$, is clopen whenever $U, V$ are clopen filters.

    An \emph{L-space morphism} is a continuous L-morphism. We write $\msf{LSpace}$ for the category of L-spaces and L-space morphisms.
\end{definition}

\begin{thm}
    The category $\msf{LSpace}$ is dually equivalent to the category $\msf{Lat}$ of lattices.
\end{thm}
\begin{proof}
    This is proved in \cite[Theorem 2.14]{bez2024}. It is obtained by restricting HMS duality to lattices and join preserving morphisms.
\end{proof}

\begin{rmk}
    Let us summarize the different categories and functors at play here (see also Figure~\ref{fig:summary}. The categories are $\msf{Lat}$, the category of lattices; $\msf{LSpace}$, the category of L-spaces and continuous L-morphisms; and $\msf{LFrm}$, the category of L-frames and L-morphisms. The functors $\msf{Fil_L}\colon\msf{Lat}\to\msf{LSpace}$ and $\msf{ClopFil}\colon\msf{LSpace}\to\msf{Lat}$ form a duality. The functor $\msf{U}\colon\msf{LSpace}\to\msf{LFrm}$ is a forgetful functor, it forgets the topology, and is faithful. The functor $\msf{Fil_F}\colon\msf{LFrm}\to\msf{Lat}$ is dually adjoint to $\msf{U}\circ\msf{Fil_L}$.
    
    \begin{figure}[ht]
        \centering
        \begin{tikzcd}[row sep=large, column sep=large]
            & \msf{LFrm} \arrow[ld, "\msf{Fil_F}"'] \\
            \msf{Lat} \arrow[r, "\msf{Fil_L}"', bend right=10] & \msf{LSpace} \arrow[u, "\msf{U}"'] \arrow[l, "\msf{ClopFil}"', bend right=10]
        \end{tikzcd}
        \caption{Summary of dualities, dual adjunctions and inclusions}
        \label{fig:summary}
    \end{figure}
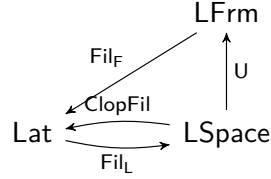
\end{rmk}

Let us recall a few properties of HMS-spaces that will come in handy later.

\begin{definition}
    An ordered topological space $(X, \preceq, \tau)$ is called a \emph{Priestley space} if $(X,\tau)$ is compact and whenever $x\not\preceq y$ then there is a clopen upset $U$ such that $x\in U$ and $y\notin U$.
\end{definition}

\begin{prop}\label{prop:priestley}
    Let $(X, \tau)$ be an HMS-space.
    \begin{enumerate}
        \item\label{item:prop:priestley:i} $X$ is a Priestley space.
        \item\label{item:prop:priestley:ii} For any closed subset $U\subseteq X$, the set ${\uparrow}U$ is closed.
        \item\label{item:prop:priestley:iii} Whenever $U$ is a closed subset of $X$ and $x\in U$, then there is a maximal point $x\preceq y$ such that $y\in U$, and a minimal point $z\preceq x$ such that $z\in U$.
        \item\label{item:prop:priestley:iv} $X$ has arbitrary meets and joins.
    \end{enumerate}
\end{prop}
\begin{proof}
    \ref{item:prop:priestley:i} follows directly by the definitions. \ref{item:prop:priestley:ii} and \ref{item:prop:priestley:iii} are well known properties of Priestley spaces (see \cite[Lemma 11.21]{dav2002} and \cite[Theorem 3.2.3]{esa2019}, and hence follow from \ref{item:prop:priestley:i}.

    For \ref{item:prop:priestley:iv}, it follows by \cite[Lemma 2.8]{bez2024} that the intersection of principal filters is principal. Given a collection $C\subseteq X$, the generator of the principal filter
    \[ \bigcap_{x\in C} {\uparrow}x \]
    is the join of $C$, denoted $\bigcurlyvee C$.
\end{proof}

We recall some more facts about this duality:

\begin{prop}\label{prop:dual_inj_surj}
\quad
    \begin{enumerate}
        \item If $f\colon X\to Y$ is a surjective L-frame morphism, then {\highlight $f^{-1}\colon \mathsf{Fil_F}(Y)\to \mathsf{Fil_F}(X)$} is an injective lattice homomorphism.
        \item If $h\colon K\to L$ is an injective lattice homomorphism, then $h^{-1}\colon\mathsf{Fil_L}(L)\to \mathsf{Fil_L}(K)$ is surjective.
    \end{enumerate}
\end{prop}
\begin{proof}
    The first claim is trivial. For the second, assume that $h$ is injective, and let $F\in\msf{Fil_L}(K)$. Observe that $h[F]$ is closed under meets, and that ${\uparrow}h[F]$ is a filter. We show that $h^{-1}[{\uparrow}h[F]] = F$.
    
    Clearly $F\subseteq h^{-1}[{\uparrow}h[F]]$. We show the converse. Let $a\in h^{-1}[{\uparrow}h[F]]$. Then there is $b\in F$ such that $h(b)\le h(a)$. That is, we have $h(a\land b) = h(a)\land h(b) = h(b)$, thus $a\land b = b$ by injectivity of $h$. This immediately implies that $b\le a$, thus $a\in F$. 
\end{proof}

Using Proposition~\ref{prop:dual_inj_surj}, we can translate amalgamation into a coamalgamation property for dual spaces. To show it, we will make use of the following key separation result.

\begin{lemma}\label{lemma:sep}
    Let $(X, \tau)$ be an L-space. Let $U,V\subseteq X$ be two closed subsets, such that $U$ and $X\setminus V$ are filters, and $U\cap V=\emptyset$. Then there exists a clopen filter $W$ such that $U\subseteq W$ and $V\subseteq X\setminus W$.
\end{lemma}
\begin{proof}
    Fix $x\in U$. For every $y\in V$, by assumption, $x\not\preceq y$, so there is some clopen filter $W_{x,y}$ such that $x\in W_{x,y}$ and $y\not\in W_{x, y}$. Hence we have $V\subseteq\bigcup_{y\in V}X\setminus W_{x,y}$, so by compactness, there are $y_1, \dots, y_n,\in V$ such that $V\subseteq (X\setminus W_{x,y_1})\cup\dots\cup(X\setminus W_{x,y_n})$. Letting $W_{x}\coloneq W_{x,y_1}\cap\dots\cap W_{x,y_{n}}$, we have $V\subseteq X\setminus W_{x}$ and $x\in W_x$, and $W_x$ is a clopen filter.
    
    Since $x\in W_{x}$, and this holds for each $x\in U$, we obtain $U\subseteq \bigcup_{x\in U}W_{x}$, so by compactness again, there are $x_1, \dots, x_n\in U$ such that $U\subseteq W_{x_1}\cup\dots\cup W_{x_n}$. Letting $W \coloneq W_{x_1}\lor\dots\lor W_{x_n}$, we have $U\subseteq W$, $W\subseteq X\setminus V$ (because $W_x\subseteq X\setminus V$ and $X\setminus V$ is a filter), and $W$ is a clopen filter, as desired.
\end{proof}

\subsection{Craig interpolation for weak positive logic}

We will now prove superamalgamation of lattices via duality. 

\begin{thm}\label{thm:supamal}
    The variety of lattices has the superamalgamation property.
\end{thm}
\begin{proof}
    Let $K,L_{1},L_{2},h_{1},h_{2}$ be a V-formation. Dually, we consider the diagram in Figure~\ref{subfig:coamal_lat}. Consider the pullback $\mrm{Pb}(f_1, f_2) \coloneq \{(x, y)\mid f_1(x)=f_2(y)\}$ together with the surjective projections $\pi_i\colon \mrm{Pb}(f_1, f_2)\to L_i$.

    \begin{figure}[ht]
        \centering
        \begin{subfigure}{.49\textwidth}
            \centering
            \begin{tikzcd}
                & \msf{Fil_F}(\mrm{Pb}(f_1, f_2)) & \\
                L_1 \arrow[ru, "\pi_1^{-1}", hook] & & L_2 \arrow[lu, "\pi_2^{-1}"', hook] \\
                & K \arrow[lu, "h_1", hook] \arrow[ru, "h_2"', hook] &
            \end{tikzcd}
            \caption{Amalgam of lattices}
            \label{subfig:amal_lat}
        \end{subfigure}
        \hfill
        \begin{subfigure}{.49\textwidth}
        \centering
            \begin{tikzcd}
                & \mrm{Pb}(f_1, f_2) \arrow[dl, "\pi_1"', two heads] \arrow[dr, "\pi_2", two heads] & \\
                Y_1 \arrow[rd, "f_1"', two heads] & & Y_2 \arrow[ld, "f_2", two heads] \\
                & X &
            \end{tikzcd}
            \caption{Coamalgam of L-frames}
            \label{subfig:coamal_lat}
        \end{subfigure}
        \caption{Amalgamation diagram and its dual}
    \end{figure}
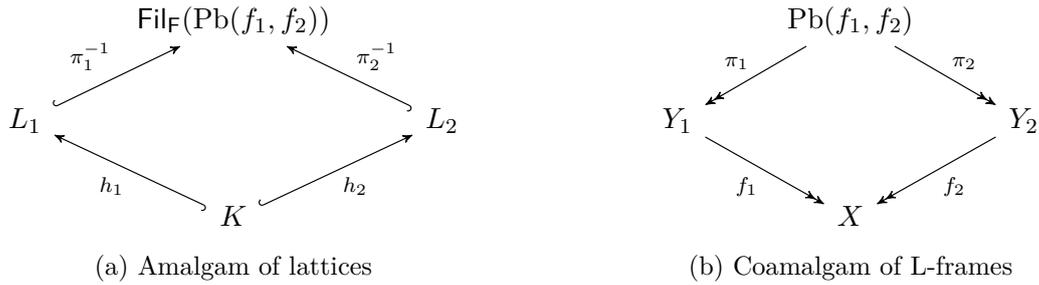
    
    Note that $\mrm{Pb}(f_{1},f_{2})$ is not guaranteed to be an L-space. However, it will always be an HMS-space, since products and closed substructures of HMS-spaces are still HMS-spaces. We show that $\pi_{i}$ are $L$-morphisms

    First assume that $\pi_1(x_1, x_2) = 1$. Then $x_1 = 1$, thus $f_2(x_2) = f_1(x_1) = 1$, and thus $x_2=1$ since $f_2$ is an L-morphism. Now assume that $y_1\curlywedge z_1\preceq\pi_1(x_1, x_2)$. Then $y_1\curlywedge z_1\preceq x_1$, thus $f_1(y_1)\curlywedge f_1(z_1)\preceq f_1(x_1) = f_2(x_2)$. Since $f_2$ is an L-morphism, we find $y_2', z_2'\in Y_2$ such that $f_1(y_1)\preceq f_2(y_2')$, $f_1(z_1)\preceq f_2(z_2')$ and $y_2'\curlywedge z_2'\preceq x_2$. Since $f_2$ is onto, we also find $y_2'', z_2''\in Y_2$ such that $f_2(y_2'') = f_1(y_1)$ and $f_2(z_2'') = f_1(z_1)$. Observe that $f_2(y_2'\curlywedge y_2'') = f_1(y_1)$, and $f_2(z_2'\curlywedge z_2'') = f_1(z_1)$, thus $(y_1, y_2'\curlywedge y_2''), (z_1, z_2'\curlywedge z_2'')\in\mrm{Pb}(f_1, f_2)$. Furthermore, we have $y_1\preceq \pi_1(y_1, y_2'\curlywedge y_2'')$, $z_1\preceq\pi_1(z_1, z_2'\curlywedge z_2'')$ and $(y_1, y_2'\curlywedge y_2'')\curlywedge(z_1, z_2'\curlywedge z_2'')\preceq(x_1, x_2)$. This proves that $\pi_1$ is an L-morphism, the proof for $\pi_2$ is similar.
    
    Hence we can consider the lattice $\msf{Fil_F}(\mrm{Pb}(f_{1},f_{2}))$, together with the lattice homomorphisms $\pi_{i}^{-1}\colon L_{i}\to \msf{Fil_F}(\mrm{Pb}(f_{1},f_{2}))$, {\highlight which are the restrictions of the morphisms $\pi_i^{-1}\colon\msf{Fil_F}(\msf{U}(Y_i))\to \msf{Fil_F}(\mrm{Pb}(f_1, f_2))$ to the sublattice $L_i = \msf{ClopFil}(Y_i)\subseteq  \msf{Fil_F}(\msf{U}(Y_i))$. By Proposition~\ref{prop:dual_inj_surj}, those homomorphisms are injective.}
    
    It is clear that the diagram commutes, so this provides an amalgam. We check that it is a superamalgam as well.
    
    We first show the following two claims:

    \begin{claim}\label{claim:supamal}
        Assume that $U\subseteq Y_{1}$ and $V\subseteq Y_{2}$ are two clopen filters.
        \begin{enumerate}
            \item ${\downarrow}f_{2}[Y_{2}\setminus V]$ is a closed subset which complement is a filter, and ${\uparrow}f_{1}[U]$ is a closed filter.
            \item If $\pi_{1}^{-1}[U]\subseteq \pi_{2}^{-1}[V]$, then ${\uparrow}f_{1}[U]\cap {\downarrow}f_{2}[Y_{2}\setminus V]=\emptyset$.
        \end{enumerate}
    \end{claim}
\begin{proof}
    The fact that ${\uparrow}f_{1}[U]$ is a filter is easy to see. To show that the complement of ${\downarrow}f_{2}[Y_{2}\setminus V]$ is a filter, take $x, y\notin {\downarrow}f_{2}[Y_{2}\setminus V]$. Suppose that $x\curlywedge y\preceq f_{2}(z)$ where $z\notin V$. Because $f_{2}$ is an L-morphism, there are $x', y'$ such that $x'\curlywedge y'\preceq z$, $x\preceq f_{2}(x')$ and $y\preceq  f_{2}(y')$. Because $x\notin {\downarrow}f_{2}[Y_2\setminus V]$, we have $x'\in V$, for otherwise, $x'\in Y_2\setminus V$ so $x\preceq f_{2}(x')$ gets us a contradiction. Similarly, $y'\in V$ and since $V$ is a filter, $x'\curlywedge y'\in V$. But $V$ is a filter, so $z\in V$, a contradiction.

    Both sets are closed. Indeed, $f_2[Y_2\setminus V]$ and $f_1[U]$ are compact, as they are images of a compact set via a continuous map, and therefore are closed. By Proposition~\ref{prop:priestley}, the sets ${\downarrow}f_{2}[Y_{2}\setminus V]$ and ${\uparrow}f_{1}[U]$ are closed.

    For the second statement, assume that $y\in {\uparrow}f_{1}[U]\cap {\downarrow}f_{2}[Y_{2}\setminus V]$. Then there are points $x\in U$ and $z\in Y_2\setminus V$ such that $f_{1}(x)\preceq y\preceq f_{2}(z)$. Now since $f_{2}$ is surjective, there is some $t$ such that $f_{1}(x)=f_{2}(t)$. Then $f_{2}(z)\curlywedge f_{2}(t)=f_{1}(x)$ and thus  $f_{2}(z\curlywedge t)=f_{1}(x)$. Hence $(x, z\curlywedge t)\in \mrm{Pb}(f_{1},f_{2})$, and this belongs to $\pi_{1}^{-1}[U]$. By assumption, this also belongs to $\pi_2^{-1}[V]$, so $z\curlywedge t\in V$, and by upwards closure, $z\in V$, a contradiction. This proves the result.
\end{proof}
    
    Now assume that $U\subseteq Y_{1}$ and $V\subseteq Y_{2}$ are clopen filters, such that $\pi_{1}^{-1}[U]\subseteq \pi_{2}^{-1}[V]$. By Lemma~\ref{lemma:sep} and Claim~\ref{claim:supamal}, we obtain that there is a clopen filter $W$ such that ${\uparrow}f_{1}[U]\subseteq W$ and ${\downarrow}f_{2}[Y_{2}\setminus V]\subseteq X\setminus W$. This  implies that $U\subseteq f_{1}^{-1}[W]$ and $f_2^{-1}[W]\subseteq V$, which proves superamalgamation.
\end{proof}

It is worth comparing this proof technique to the one exposed in \cite{jip1992}, which is originally due to Jónsson \cite{jon1956}. We recall the proof technique used there.

\begin{proof}[{\highlight Another} proof of Theorem~\ref{thm:supamal}]
    Let $K,L_{1},L_{2},h_{1},h_{2}$ be a V-formation. Without loss of generality, we may assume that $K$ is a sublattice of both $L_1$, that $L_2$, and $L_1\cap L_2 = K$ and that $h_1$, $h_2$ are inclusions, which we denote by $\iota_1$, $\iota_2$. We have to show that there is a lattice $M$ with embeddings $p_1\colon L_1\to M$ and $p_2\colon L_2\to M$ such that $M$ is a superamalgam.

    Let $\le_1$ be the order on $L_1$ and $\le_2$ be the order on $L_2$. We define an order $\le$ on $L_1\cup L_2$ as the transitive closure of ${\le_1}\cup{\le_2}$, turning $L_1\cup L_2$ into a poset. One easily shows that $a\le b$ iff $a, b\in L_i$ and $a\le_i b$, or $a\in L_i$, $b\in L_j$ and there is $c\in K$ such that $a\le_ic\le_jb$.

    A subset $F$ of $L_1\cup L_2$ is called an \emph{$(L_1, L_2)$-filter} if it is an upset and $a, b\in F\cap L_i$ implies $a\land_ib\in F$, where $\land_i$ is the meet in $L_i$. We denote by $\mrm{Fil}(L_1, L_2)$ the set of $(L_1, L_2)$-filters of $L_1\cup L_2$, ordered by reverse inclusion. This is a lattice, and we have embeddings $p_i\colon L_i\to \mrm{Fil}(L_1, L_2) \quad a\mapsto {\uparrow}a$. One easily shows that this is a superamalgam.

    \vspace{3mm}

    The set $\mrm{Fil}(L_1, L_2)$ thus defined is the order dual of the set $\mrm{Pb}(f_1, f_2)$ defined in the first proof. Indeed, when the maps $h_i\colon K\to L_i$ are inclusions, the maps $f_i\colon Y_i\to X$ are are restrictions sending $F\subseteq Y_i$ to $F\cap K$. Therefore, the elements of $\mrm{Pb}(f_1, f_2)$ are pairs of filters $(F_1, F_2)$ such that $F_1\cap K = F_2\cap K$. One easily verifies that if $(F_1, F_2)$ is such a pair, then $F_1\cup F_2$ is an $(L_1, L_2)$-filter; and that conversely, if $F$ is an $(L_1, L_2)$-filter, then $(F\cap L_1, F\cap L_2)$ is an element of $\mrm{Pb}(f_1, f_2)$; those two constructions being each other's inverse. As $\mrm{Fil}(L_1, L_2)$ is ordered by reverse inclusion, while $\mrm{Pb}(f_1, f_2)$ is ordered by pairwise inclusion, their orders are reversed.
\end{proof}

\begin{rmk}
    While Jónsson's proof is striking by its simplicity, it should be noted that it does not straightforwardly generalize to the modal case. This can be explained by the fact that the amalgam $\mrm{Fil}(L_1, L_2)$ actually belongs to the dual side. Therefore, if we were to add modalities to the initial lattices, this would correspond to adding a relation on $\mrm{Fil}(L_1, L_2)$, but it is unclear how this could induce modalities on that same lattice.

    On the other hand, we will see in the next section that the method presented in this paper generalizes to the modal case, with the modalities on the initial lattices corresponding to relations on the dual L-frame $\mrm{Pb}(f_1, f_2)$, which in turn can be turned into modalities on $\msf{Fil_F}(\mrm{Pb}(f_1, f_2))$
\end{rmk}

\section{Superamalgamation for modal lattices}\label{section:supamal_modal}

\subsection{Duality for modal lattices}\label{subsec:dual}

In this section we look at superamalgamation of modal lattices, using the general strategy employed in the previous section. For that purpose, we extend the duality used in the previous section.  Let us first recall the modal version of L-frames, first introduced in \cite{bez2024}.

\begin{definition}\label{def:modal_L-frame}
    A \emph{modal L-frame} is a tuple $(X, 1, \curlywedge, R)$ where $(X, 1, \curlywedge)$ is an L-frame, and $R$ is a binary relation on $X$ such that:
    \begin{enumerate}
        \item\label{item:def:modal_L-frame:i} if $x\preceq y$ and $y\mrel{R}z$, there is a $w\in X$ such that $x\mrel{R}w$ and $w\preceq z$;
        \item\label{item:def:modal_L-frame:ii} if $x\preceq y$ and $x\mrel{R}w$, there is a $z\in X$ such that $y\mrel{R}z$ and $w\preceq z$;
        \item\label{item:def:modal_L-frame:iii} if $(x\curlywedge y)\mrel{R}z$, there are $u, v\in X$ such that $x\mrel{R}u$, $y\mrel{R}v$ and $u\curlywedge v\preceq z$;
        \item\label{item:def:modal_L-frame:iv} if $x\mrel{R}u$ and $y\mrel{R}v$, then $(x\curlywedge y)\mrel{R}(u\curlywedge v)$;
        \item\label{item:def:modal_L-frame:v} for all $x\in X$, $1\mrel{R}x$ iff $x=1$.
    \end{enumerate}

    \begin{figure}[ht]
        \centering
        \begin{tikzcd}[row sep=large,
                       column sep=large,
                       %every matrix/.append style={draw, inner ysep=2pt, 
                       %inner xsep=3pt, rounded corners}
                       ]
            \text{
            \begin{tikzpicture}[scale=.9, baseline=0]
                \node (x) at (0,0) {$x$};
                \node (y) at (0,1.2) {$y$};
                \node (z) at (1.4,0) {$w$};
                \node (w) at (1.4,1.2) {$z$};
                \draw (x) to (y);
                \draw[densely dashed, -stealth'] (x) to node[below]{\scriptsize{$R$}} (z);
                \draw[-stealth'] (y) to node[above]{\scriptsize{$R$}} (w);
                \draw[densely dashed] (w) to (z);
                \node (label) at (.7,-1) {\ref{item:def:modal_L-frame:i}};
            \end{tikzpicture}
            \hspace{1.5em}
            \begin{tikzpicture}[scale=.9, baseline=0]
                \node (x) at (0,0) {$x$};
                \node (y) at (0,1.2) {$y$};
                \node (z) at (1.4,0) {$w$};
                \node (w) at (1.4,1.2) {$z$};
                \draw (x) to (y);
                \draw[-stealth'] (x) to node[below]{\scriptsize{$R$}} (z);
                \draw[densely dashed, -stealth'] (y) to node[above]{\scriptsize{$R$}} (w);
                \draw[densely dashed] (w) to (z);
                \node (label) at (.7,-1) {\ref{item:def:modal_L-frame:ii}};
            \end{tikzpicture}
            \hspace{1.5em}
            \begin{tikzpicture}[scale=.9, baseline=0]
                \node (x) at (-.9,1) {$x$};
                \node (y) at (.4,1.5) {$y$};
                \node (xandy) at (0,0) {$x \curlywedge y$};
                \node (z) at (2.2,0) {$z$};
                \node (v) at (1.4,1) {$u$};
                \node (w) at (2.8,1.5) {$v$};
                \node (vandw) at (2.3,-1) {$u \curlywedge v$};
                \node (vandwleft) at (2.1,-.9) {};
                \node (vandwright) at (2.5,-.9) {};
                \draw[-stealth'] (xandy) to node[below,pos=.4]{\scriptsize{$R$}} (z);
                \draw (x) to (xandy);
                \draw (y) to (xandy);
                \draw[white, line width=5pt, bend right=12] (v) to (vandwleft);
                \draw[densely dashed, bend right=12] (v) to (vandwleft);
                \draw[densely dashed, bend left=10] (w) to (vandwright);
                \draw[white, line width=5pt] (x) to (v);
                \draw[densely dashed, -stealth'] (x) to node[below,pos=.7]{\scriptsize{$R$}} (v);
                \draw[densely dashed, -stealth'] (y) to node[above]{\scriptsize{$R$}} (w);
                \draw[densely dashed] (vandw) to (z);
                \node (label) at (1,-1) {\ref{item:def:modal_L-frame:iii}};
            \end{tikzpicture}
            \hspace{1.5em}
            \begin{tikzpicture}[scale=.9, baseline=0]
                \node (x) at (-.9,1) {$x$};
                \node (y) at (.4,1.5) {$y$};
                \node (xandy) at (0,0) {$x \curlywedge y$};
                \node (v) at (1.4,1) {$u$};
                \node (w) at (2.7,1.5) {$v$};
                \node (vandw) at (2.3,0) {$u \curlywedge v$};
                \draw (x) to (xandy);
                \draw (y) to (xandy);
                \draw[white, line width=5pt] (x) to (v);
                \draw[-stealth'] (x) to node[below,pos=.75]{\scriptsize{$R$}}(v);
                \draw[-stealth'] (y) to node[above]{\scriptsize{$R$}} (w);
                \draw (v) to (vandw);
                \draw (w) to (vandw);
                \draw[densely dashed, -stealth'] (xandy) to node[below]{\scriptsize{$R$}} (vandw);
                \node (label) at (1.2,-1) {\ref{item:def:modal_L-frame:iv}};
            \end{tikzpicture}
            }
        \end{tikzcd}
        \caption{The first four conditions of a modal L-frame.
               Lines denote the poset order, with high nodes being bigger.
               Arrows denote the relation $R$.}
        \label{fig:modal_L-frame}
    \end{figure}
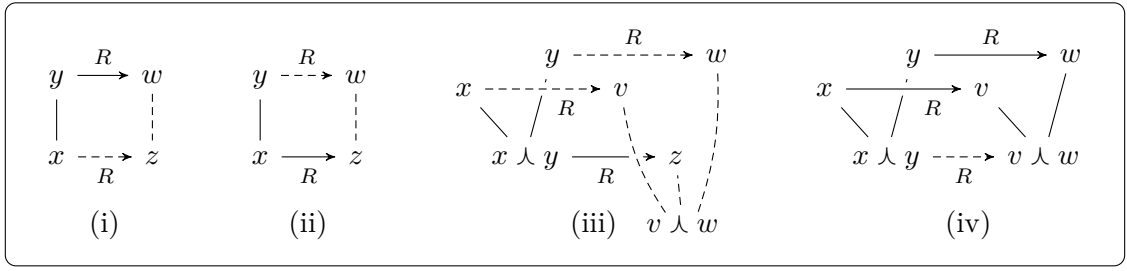
\end{definition}

Any modal L-frame $(X, R)$ gives rise to modal operators on the lattice $\msf{Fil_F}(X)$. Indeed, the five conditions ensure that $\Box_R\colon U\mapsto \{ x\in X\mid R[x]\subseteq U\}$ and $\Diamond_R\colon U\mapsto \{ x\in X\mid R[x]\cap U\neq\emptyset \}$ send filters to filters.

\begin{definition}\label{def:bounded_morph}
    A \emph{bounded L-morphism} between modal L-frames $(X, R)$ and $(Y, R)$ is an L-morphism $f\colon X\to Y$ such that:
    \begin{enumerate}
        \item if $x\mrel{R}y$, then $f(x)\mrel{R}f(y)$;
        \item if $f(x)\mrel{R}z$, then there is $y\in X$ such that $x\mrel{R}y$ and $f(y)\preceq z$;
        \item if $f(x)\mrel{R}z$, then there is $y\in X$ such that $x\mrel{R}y$ and $z\preceq f(y)$.
    \end{enumerate}
\end{definition}

One easily checks that if $f\colon(X, R)\to (Y, R)$ is a bounded L-morphism, then $f^{-1}\colon \msf{Fil_F}(Y)\to\msf{Fil_F}(X)$ preserves $\Box_R$ and $\Diamond_R$.

Now, to obtain a duality, a topology needs to be added to modal L-frames. This gives the following definition.

\begin{definition}
    A \emph{modal L-space} is a tuple $(X, 1, \curlywedge, \tau, R)$ such that:
    \begin{enumerate}
        \item $(X, 1, \curlywedge, \tau)$ is an L-space;
        \item $R$ is a binary relation on $X$ such that $1\mrel{R}x$ iff $x=1$;
        \item if $U$ is a clopen filter, then so are $\Box_R U$ and $\Diamond_R U$;
        \item for all $x, y\in X$, we have $x\mrel{R}y$ iff for all clopen filter $U$, $x\in \Box_RU$ implies $y\in U$ and $y\in U$ implies $x\in\Diamond_RU$.
    \end{enumerate}

    A \emph{modal L-space morphism} is a continuous bounded L-morphism.

    We write $\msf{MLSpace}$ for the category of modal L-spaces and modal L-space morphisms.
\end{definition}

If $(X, \tau, R)$ is a modal L-space, then $(X, R)$ is a modal L-frame (see \cite[Lemma 4.15]{bez2024}). We denote the underlying L-frame of a modal L-space by $\msf{U}(X)$. Let us recall some basic facts about modal L-spaces.

\begin{thm}
    The category $\msf{MLSpace}$ is dually equivalent to the category $\msf{MLat}$ of modal lattices.
\end{thm}
\begin{proof}[Proof]
    See \cite[Theorem 4.26]{bez2024} for a detailed proof. Here is a quick sketch.
    
    As noted in Theorem~\ref{thm:HMS_duality}, the functors $\msf{Fil_L}$ and $\msf{ClopFil}$ send lattices to L-spaces and vice versa. Given a modal lattice $(L, \Box, \Diamond)$, we define a relation $R$ on $\msf{Fil_L}(L)$ by
    \[ x\mrel{R}y\qquad\text{iff}\qquad\text{for all}~a\in L,~\Box a\in x~\text{implies}~a\in y~\text{and}~a\in y~\text{implies}~\Diamond a\in x. \]
    This turns $\msf{Fil_L}(L)$ into a modal L-space.
    
    In turn, given a modal L-space $(X, \tau, R)$, the conditions on $R$ ensure that $\Box_R$ and $\Diamond_R$ are well defined modal operators. Moreover, one easily observes that if an L-space morphism $f\colon(X, \tau, R)\to (Y, \tau, R)$ is a bounded morphism, then its dual $f^{-1}\colon\msf{ClopFil}(Y)\to\msf{ClopFil}(X)$ preserves modalities. Similarly, if $h\colon(L, \Box, \Diamond)\to(L, \Box, \Diamond)$ is a modal lattice homomorphism, then $h^{-1}\colon\msf{Fil_L}(L)\to\msf{Fil_L}(K)$ is a bounded morphism.
\end{proof}

\begin{prop}\label{prop:succ_set}
    Let $(X, \tau, R)$ be a modal L-space. Then $R[x]$ is a closed nonempty convex set for all $x\in X$. Moreover, for any $x\in X$, the set $R[x]$ is closed under meets.
\end{prop}
\begin{proof}
    The fact that {\highlight $R[x]$} is nonempty follows from conditions \ref{item:def:modal_L-frame:i} and \ref{item:def:modal_L-frame:v} of Definition~\ref{def:modal_L-frame}.

    {\highlight The fact that $R[x]$ is a closed is proved in \cite[Lemma 4.15]{bez2024}. In order to make the proof more self-contained we include the argument here.} Assume that $y\not\in R[x]$. Then there is a clopen filter $F$ such that $x\in\Box_RF$ and $y\not\in F$, or such that $y\in F$ and $x\not\in \Diamond_R F$. In the former case, we have $R[x]\subseteq F$ and $y\not\in F$, so $F^c$ is an open neighborhood of $y$ disjoint from $R[x]$. In the latter case, we have $R[x]\cap F=\emptyset$ and $y\in F$, so $F$ is an open neighborhood of $y$ disjoint from $R[x]$.

    {\highlight The last statement} follows from the definition of a modal L-space. Indeed, assume that $y\preceq z \preceq  t$ and $y, t\in R[x]$ but $z\not\in R[x]$. Then there is a clopen filter $F$ such that $x\in\Box_R F$ and $z\not\in F$ or such that $z\in F$ but $x\not\in\Diamond_R F$. In the former case, we have $R[x]\subseteq F$, thus $y\in F$, which contradicts $z\not\in F$ as $y\preceq z$. In the latter case, we have $z\in F$ and $R[x]\cap F = \emptyset$, which contradicts $t\in F$ as $t\in R[x]$ and $t\succeq  z\in F$.

    The last statement follows from the fact that $y, z\in R[x]$ means $x\mrel{R}y$ and $x\mrel{R}z$, thus $x = x\curlywedge x\mrel{R} y\curlywedge z$ by condition \ref{item:def:modal_L-frame:iv} of Definition~\ref{def:modal_L-frame}.
\end{proof}

\begin{coro}\label{coro:succ_has_max}
    Let $(X, \tau, R)$ be an L-space. Then for any $x, y\in X$ such that $x\mrel{R}y$, we can find maximal and minimal elements $z, t\in R[x]$ such that $z\preceq y\preceq t$.
\end{coro}
\begin{proof}
    This follows from Propositions~\ref{prop:priestley} and \ref{prop:succ_set} straightforwardly.
\end{proof}

\subsection{Craig interpolation for weak positive modal logic}

We are now ready to prove the main result of this paper. 

\begin{thm}
    The variety of modal lattices has the superamalgamation property.
\end{thm}
\begin{proof}
    We proceed as in the proof of Theorem~\ref{thm:supamal}. We must define a relation on $\mrm{Pb}(f_1, f_2)$, show that it gives rise to a modal L-frame, and that the projection maps $\pi_i$ are bounded L-morphisms. This allows us to define modalities on $\msf{Fil_F}(\mrm{Pb}(f_1, f_2))$, and shows that the maps $\pi_i^{-1}$ preserve those modalities.

    We define the relation $R$ on $\mrm{Pb}(f_1, f_2)$ as $(x_1, x_2)\mrel{R}(y_1, y_2)$ iff $x_1\mrel{R}y_1$ and $x_2\mrel{R}y_2$. Let us first show that it satisfies the conditions for a L-frame {\highlight in Definition~\ref{def:modal_L-frame}.}

    \begin{enumerate}\setlength{\parindent}{1em}
        \item Assume $(x_1, x_2)\preceq (y_1, y_2)$ and $(y_1, y_2)\mrel{R}(z_1, z_2)$. Let $x \coloneq f_1(x_1) = f_2(x_2)$, $y \coloneq f_1(y_1) = f_2(y_2)$ and $z \coloneq f_1(z_1) = f_2(z_2)$. We have $x\preceq y\mrel{R}z$, thus there is $t\in X$  such that $x\mrel{R}t\preceq z$. By Corollary~\ref{coro:succ_has_max}, we may assume that $t$ is minimal in $R[x]$.

        We have $f_1(x_1)\mrel{R}t$, and since $f_1$ is a bounded L-morphism, we can find $t'_1\in Y_1$ such that $x_1\mrel{R}t'_1$ and $f_1(t'_1)\preceq t$. Because $Y_{1}$ satisfies condition \ref{item:def:modal_L-frame:i} of a modal L-frame, we can also find $t''_1\in Y_1$ such that $x_1\mrel{R}t''_1\preceq z_1$.  
        
        Let $t_1 \coloneq t'_1\curlywedge t''_1$. Then $x_1\mrel{R}t_1$ (by condition \ref{item:def:modal_L-frame:iv} of a modal L-frame), $t_1\preceq z_1$ and $f_1(t_1)\preceq t$. However, as $f_1$ preserves $R$, we have $f_1(t_1)\in R[x]$, thus $f_1(t_1) = t$ by minimality of $t$.
        
        Similarly, we find $t_2\in Y_2$ such that $x_2\mrel{R}t_2$, $t_2\preceq z_2$ and $f_2(t_2) = t$. Then, the pair $(t_1, t_2)$ is in $\mrm{Pb}(f_1, f_2)$, and we have $(x_1, x_2)\mrel{R}(t_1, t_2)\preceq (z_1, z_2)$.

        \item We proceed similarly.

        \item Assume $(x_1, x_2)\curlywedge(y_1, y_2)\mrel{R}(z_1, z_2)$. Let $x \coloneq f_1(x_1) = f_2(x_2)$, $y \coloneq f_1(y_1) = f_2(y_2)$ and $z \coloneq f_1(z_1) = f_2(z_2)$. We have $x\curlywedge y\mrel{R}z$, thus there are $u, v\in X$ such that $x\mrel{R}u$, $y\mrel{R}v$ and $u\curlywedge v\preceq z$. By Corollary~\ref{coro:succ_has_max}, we may assume that $u$ is minimal in $R[x]$, and $v$ is minimal in $R[y]$.

        We have $f_1(x_1)\mrel{R} u$, and $f_1(y_1)\mrel{R}v$, and since $f_1$ is a bounded L-morphism, we can find $u'_1, v'_1\in Y_1$ such that $x_1\mrel{R}u'_1$, $y_1\mrel{R}v'_1$, $f_1(u'_1)\preceq u$ and $f_1(v'_1)\preceq v$. Since $x_1\curlywedge y_1\mrel{R}z_1$, we can also find $u''_1, v''_1\in Y_1$ such that $x_1\mrel{R}u''_1$, $y_1\mrel{R}v''_1$ and $u''_1\curlywedge v''_1\preceq z_1$. Let $u_1 \coloneq u'_1\curlywedge u''_1$, and $v_1 \coloneq v'_1\curlywedge v''_1$. Then $x_1\mrel{R}u_1$, $y_1\mrel{R}v_1$, $u_1\curlywedge v_1\preceq z_1$, $f_1(u_1)\preceq u$ and $f_1(v_1)\preceq v$. From $f_1(u_1)\in R[x]$, $f_1(u_1)\preceq u$ and the minimality of $u$, we get $f_1(u_1) = u$. Similarly, we get $f_1(v_1) = v$.

        Proceeding similarly, we obtain $u_2, v_2\in Y_2$ such that $x_2\mrel{R}u_2$, $y_2\mrel{R}v_2$, $u_2\curlywedge v_2\preceq z_2$, $f_2(u_2) = u$ and $f_2(v_2) = v$. Then, the pairs $(u_1, u_2)$ and $(v_1, v_2)$ are in $\mrm{Pb}(f_1, f_2)$, and we have $(x_1, x_2)\mrel{R}(u_1, u_2)$, $(y_1, y_2)\mrel{R}(v_1, v_2)$ and $(u_1, u_2)\curlywedge(v_1, v_2)\preceq (z_1, z_2)$.
    \end{enumerate}

    The last two conditions are easy to verify, since everything is defined componentwise.
    
    Now we show that the projections $\pi_i$ are bounded L-morphisms. {\highlight From Theorem~\ref{thm:supamal}, we know that these are L-morphisms, so it is sufficient to prove that they are bounded.} We do it for $\pi_1$, showing that it satisfies all three conditions {\highlight in Definition~\ref{def:bounded_morph}}.

    \begin{enumerate}\setlength{\parindent}{1em}
        \item Follows straightforwardly from the definition of $R$ on $\mrm{Pb}(f_1, f_2)$.
        
        \item Assume that $\pi_1(x_1, x_2)\mrel{R}z_1$, that is,  $x_1\mrel{R}z_1$. Let $x \coloneq f_1(x_1) = f_2(x_2)$ and $z \coloneq f_1(z_1)$. Then $x\mrel{R}z$, and we can find a $t\preceq z$ that is minimal in $R[x]$. As $f_1(x_1)=f_2(x_2)\mrel{R}t$, we can find $y_1\in Y_1$ and $y_2\in Y_2$ such that $x_1\mrel{R}y_1$, $x_2\mrel{R}y_2$, $f_1(y_1)\preceq t$ and $f_2(y_2)\preceq t$. {\highlight Then $y_1\curlywedge z_1\in R[x_1]$, as $x_1\mrel{R}y_1$ and $x_1\mrel{R}z_1$. Therefore, we have $f_1(y_1\curlywedge z_1), f_2(y_2)\in R[x]$, and as $t$ is minimal, $f_1(y_1\curlywedge z_1) = t = f_2(y_2)$. Thus the pair $(y_1\curlywedge z_1, y_2)$ is in $\mrm{Pb}(f_1, f_2)$, and we have $(x_1, x_2)\mrel{R}(y_1\curlywedge, y_2)$ and $\pi_1(y_1\curlywedge z_1, y_2)=y_1\curlywedge z_1\preceq z_1$.}
        
        \item We proceed similarly to the previous case. \qedhere
    \end{enumerate}
\end{proof}

Combining this with Theorem~\ref{thm:inter_amal}, we obtain the following.

\begin{coro}
    Weak positive modal logic has the Craig interpolation property.
\end{coro}

\subsection{Craig interpolation in extensions}

The technique used to prove the previous result generalizes to any variety $\msf{K}$ of modal lattices such that there is a class $\msf{C}$ of modal L-frames satisfying $\msf{U}[\msf{Fil_L}[\msf{K}]]\subseteq\msf{C}$, $\msf{Fil_F}[\msf{C}]\subseteq\msf{K}$, and such that for any co-V-formation in $\msf{Fil_L}[\msf{K}]$, the coamalgam $\mrm{Pb}$ is in $\msf{C}$ (recall that the functor $\mathsf{Fil_L}$ turns modal lattices into modal $L$-space).

In particular, if a logic $\mcal{L}(\Gamma)$ is sound and complete with respect to a class of L-frames axiomatized by universal Horn formulas, then the above condition is satisfied for $\msf{K}(\Gamma)$---proving that $\msf{K}(\Gamma)$ has superamalgamation and that $\mcal{L}(\Gamma)$ has Craig interpolation.

We start by recalling from \cite{bez2024} the L-frame semantics for our logic.

\begin{definition}
    A \emph{modal L-model} is a pair $(X, R, V)$ where $(X, R)$ is a modal L-frame and $V\colon P\to\msf{Fil_F}(X)$ is a valuation. The interpretation of formulas at a state $x$ of a modal L-model $(X, R, V)$ is defined as follows:
\begin{alignat*}{2}
        (X, R, V), x&\Vdash\top && \text{always}, \\
        (X, R, V), x&\Vdash\bot &\quad\text{iff}\quad& x=1, \\
        (X, R, V), x&\Vdash p &\quad\text{iff}\quad& x\in V(p), \\
        (X, R, V), x&\Vdash\phi\land\psi &\quad\text{iff}\quad& (X, V), x\Vdash\phi~\text{and}~(X, V), x\Vdash\psi, \\
        (X, R, V), x&\Vdash\phi\lor\psi &\quad\text{iff}\quad& \exists y, z\in X~\text{such that}~(X, V), y\vdash\phi~\text{and}~(X, V), z\Vdash\psi~\text{and}~y\curlywedge z\preceq x, \\
        (X, R, V), x&\Vdash\Box\phi &\quad\text{iff}\quad& \forall y\in X,~x\mrel{R}y~\text{implies}~(X, V), y\Vdash\phi, \\
        (X, R, V), x&\Vdash\Diamond\phi &\quad\text{iff}\quad& \exists y\in X~\text{such that}~x\mrel{R}y~\text{and}~(X, V), y\Vdash\phi.
    \end{alignat*}

    Observe that the conditions imposed on a modal L-frame ensure that the truth set $V(\phi) \coloneq \{x\in X\mid (X, R, V), x\Vdash\phi\}$ of any formula $\phi$ is a filter. Thus the valuation extends to a map $V\colon Fm_P\to\msf{Fil_F}(X)$.
\end{definition}

\begin{definition}
    We say that a modal L-frame $(X, R)$ \emph{validates} a consequence pair $\phi\trianglelefteq\psi$, written $(X, R)\Vdash\phi\trianglelefteq\psi$, if for every valuation $V$ on $X$, we have $V(\phi)\subseteq V(\psi)$.

    A class $\msf{C}$ of L-frames \emph{validates} a consequence pair $\phi\trianglelefteq\psi$, written $\Vdash_\msf{C}\phi\trianglelefteq\psi$, if for all $(X, R)\in\msf{C}$, we have $(X, R)\Vdash\phi\trianglelefteq\psi$.
\end{definition}

\begin{definition}
    We say that a logic $\mcal{L}(\Gamma)$ is \emph{sound and complete} with respect to a class $\msf{C}$ of L-frames if for any consequence pair $\phi\trianglelefteq\psi$, we have $\Vdash_\msf{C}\phi\trianglelefteq\psi$ iff $\phi\vdash_\Gamma\psi$. 
\end{definition}

The following {\highlight two propositions are} a straightforward adaptation of the classical case:

\begin{prop}
    Given a logic $\mcal{L}(\Gamma)$ and a class $\msf{C}$ of L-frames, $\mcal{L}(\Gamma)$ is sound and complete with respect to $\msf{C}$ iff $\msf{U}[\msf{Fil_L}[\msf{K}(\Gamma)]]\subseteq\msf{C}$ and $\msf{Fil_F}[\msf{C}]\subseteq\msf{K}(\Gamma)$.
\end{prop}

\begin{definition}
    Given a modal L-space $(X, \tau, R)$, a \emph{clopen valuation} is an assignment $V\colon P\to\msf{ClopFil}(X)$. A \emph{modal L-space model} is a pair $(X, \tau, R, V)$ where $(X, \tau, R)$ is a modal L-space and $V$ a clopen valuation. Formulas are interpreted on a modal L-space model $(X, V)$ as on the underlying modal L-model $({\msf{U}(X)}, V)$.

    Observe that the conditions imposed on a modal L-space ensure that the truth set $V(\phi)$ of any formula $\phi$ is a clopen filter. Thus we have a valuation $V\colon Fm_P\to\msf{ClopFil}(X)$.

    A modal L-space $(X, \tau, R)$ validates a consequence pair $\phi\trianglelefteq\psi$, denoted $(X, \tau, R)\Vdash\phi\trianglelefteq\psi$, if for all clopen valuation $V$, we have $V(\phi)\subseteq V(\psi)$.
\end{definition}

\begin{prop}
    Given a logic $\mcal{L}(\Gamma)$, the class of modal L-spaces validating every consequence pair in $\Gamma$ is exactly $\msf{Fil_L}[\msf{K}(\Gamma)]$.
\end{prop}

We will now introduce some natural extensions of weak positive modal logic.

\begin{definition}\label{def:axioms}
    We {\highlight recall} the following axioms:
    \begin{alignat*}{2}
        \msf{T} &\coloneq\Box p\trianglelefteq p, &&p\trianglelefteq\Diamond p, \\
        \msf{4} &\coloneq\Box p\trianglelefteq\Box\Box p, &&\Diamond\Diamond p\trianglelefteq\Diamond p, \\
        \msf{B} &\coloneq p\trianglelefteq\Box\Diamond p, &&\Diamond\Box p\trianglelefteq p, \\
        \msf{5} &\coloneq\Diamond p\trianglelefteq\Box\Diamond p, &&\Diamond\Box p\trianglelefteq\Box p, \\
        \msf{.2} &\coloneq\Diamond\Box p\trianglelefteq\Box\Diamond p. 
    \end{alignat*}
\end{definition}

\begin{definition}
    These axioms correspond to first order relational properties:
    \begin{description}[font=\normalfont\textit]
        \item[reflexivity] $\forall x,~(x\mrel{R}x)$,
        \item[transitivity] $\forall x, y, z,~(x\mrel{R}y~\text{and}~y\mrel{R}z\rightarrow x\mrel{R}z)$,
        \item[symmetry] $\forall x, y,(x\mrel{R}y\rightarrow y\mrel{R}x)$,
        \item[Euclideanity] $\forall x, y, z,~(x\mrel{R}y~\text{and}~x\mrel{R}z\rightarrow y\mrel{R}z)$,
        \item[directedness] $\forall x, y, z,~(x\mrel{R}y~\text{and}~x\mrel{R}z\rightarrow\exists t~(y\mrel{R}t~\text{and}~z\mrel{R}t))$.
    \end{description}
\end{definition}

The following theorem extends well known results for classical modal logic.

\begin{thm}
    Let $\Gamma\subseteq\{\msf{T}, \msf{4}, \msf{B}, \msf{5}, \msf{.2}\}$, and let $\Theta_\Gamma$ be the corresponding set of first order relational properties.
    \begin{enumerate}
        \item The class of L-spaces axiomatized by $\Theta_\Gamma$ is exactly $\msf{Fil_L}[\msf{K}(\Gamma)]$.
        \item The logic $\mcal{L}(\Gamma)$ is sound and complete with respect to the class $\msf{C}$ of L-frames axiomatized by $\Theta_\Gamma$.
    \end{enumerate}
\end{thm}
\begin{proof}
    We first show that a modal L-space $(X, \tau, R)$ validating every consequence pair in $\Gamma$ also satisfies the conditions in $\Theta_\Gamma$. By \cite[Lemma 4.33]{bez2024}, the underlying L-frame ${\msf{U}(X)}$ validates the consequence pairs in $\Gamma$.
    
    By the Sahlqvist algorithm described in \cite[Theorem 4.39]{bez2024}, we know that $(X, R)$ satisfies
    \[ \forall x~\exists y_1~(x\mrel{R}y_1~\text{and}~y_1\preceq x) \qquad\text{and}\qquad \forall x~\exists y_2~(x\mrel{R}y_2~\text{and}~x\preceq y_2) \]
    if $\msf{T}\in\Gamma$,
    \begin{multline*}
        \forall x, y, z~(x\mrel{R}y\mrel{R}z\rightarrow\exists t_1~(x\mrel{R}t_1~\text{and}~t_1\preceq z)) \qquad\text{and} \\ \forall x, y, z~(x\mrel{R}y\mrel{R}z\rightarrow\exists t_2~(x\mrel{R}t_2~\text{and}~z\preceq t_2)).
    \end{multline*}
    if $\msf{4}\in\Gamma$,
    \[ \forall x, y~(x\mrel{R}y\rightarrow\exists z_1~(y\mrel{R}z_1~\text{and}~x\preceq z_1)) \qquad\text{and}\qquad \forall x, y~(x\mrel{R}y\rightarrow\exists z_2~(y\mrel{R}z_2~\text{and}~z_2\preceq x)) \]
    if $\msf{B}\in\Gamma$,
    \begin{multline*}
        \forall x, z, y~(x\mrel{R}z~\text{and}~x\mrel{R}y\rightarrow\exists t_1~(y\mrel{R}t_1~\text{and}~z\preceq t_1)) \qquad\text{and} \\ \forall x, y, z~(x\mrel{R}y~\text{and}~x\mrel{R}z\rightarrow\exists t_2~(y\mrel{R}t_2~\text{and}~t_2\preceq z))
    \end{multline*}
    if $\msf{5}\in\Gamma$ (note that we swapped $y$ and $z$ in the first formula), and
    \begin{equation}\label{eqn:direct}
        \forall x, y, z~(x\mrel{R}y~\text{and}~x\mrel{R}z\rightarrow\exists u, v~(y\mrel{R}u~\text{and}~z\mrel{R}v~\text{and}~u\preceq v))
    \end{equation}
    if $\msf{.2}\in\Gamma$.

    For the first four {\highlight axioms}, recall that in a modal L-space, the set $R[x]$ is always convex (that is, if $x\mrel{R}t_1, x\mrel{R}t_2$, and $t_1\preceq y\preceq t_2$, then $x\mrel{R}y$). The result follows. The axiom $\msf{.2}$ requires a bit more work.

    {\highlight Take $x, y, z\in X$ such that $x\mrel{R}y$ and $x\mrel{R}z$. By \eqref{eqn:direct}, there are $u\in R[y], v\in R[z]$ such that $u\preceq v$. Swapping $y$ and $z$ in \eqref{eqn:direct}, we also get $w\in R[y], t\in R[z]$ such that $t\preceq w$. We have
    \[ u\curlywedge w\preceq v\curlywedge w\preceq w \qquad\text{and}\qquad v\curlywedge t\preceq v\curlywedge w\preceq v, \]
    and as $u\curlywedge w, w\in R[y]$ and $v\curlywedge t, v\in R[z]$, we get $v\curlywedge w\in R[y]\cap R[z]$ by convexity. This proves directedness.}

    This finishes the proof that every modal L-space in $\msf{Fil_L}[\msf{K}(\Gamma)]$ satisfies the conditions in $\Theta_\Gamma$. It follows that $\msf{U}[\msf{Fil_L}[\msf{K}(\Gamma)]]\subseteq\msf{C}$. Checking that any modal L-frame in $\msf{C}$ validates the consequence pairs in $\Gamma$ is standard, and it follows immediately that any modal L-space satisfying the conditions in $\Theta_\Gamma$ validates the consequence pairs in $\Gamma$ (as $(X, \tau, R)$ validates a consequence pair if $\msf{U}(X)$ does).
\end{proof}

\begin{rmk}
    Note that the class $\msf{C}$ of modal L-frames introduced in the previous theorem need not be the largest class of modal L-frame that $\mcal{L}(\Gamma)$ is sound and complete for.
\end{rmk}

{\highlight We can now prove that the varieties of modal lattices axiomatized by a set of axioms recalled in Definition~\ref{def:axioms} have the superamalgamation property.}

\begin{thm}\label{thm:supamal_ext}
    Let $\Gamma\subseteq\{\msf{T}, \msf{4}, \msf{B}, \msf{5}, \msf{.2}\}$. Then $\msf{K}(\Gamma)$ has the superamalgamation property
\end{thm}
\begin{proof}
    Let $\Theta_\Gamma$ be the set of first order conditions corresponding to $\Gamma$, and let $\msf{C}$ be the class of modal L-frames axiomatized by $\Theta_\Gamma$. We already know that $\msf{U}[\msf{Fil_L}[\msf{K}(\Gamma)]]\subseteq\msf{C}$, and that $\msf{Fil_F}[\msf{C}]\subseteq \msf{K}(\Gamma)$. It is sufficient to show that given a co-V-formation of L-spaces $X, Y_1, Y_2$ and $f_1\colon Y_1\to X$, $f_2\colon Y_2\to X$, the L-frame $\mrm{Pb}(f_1, f_2)$ is in $\msf{C}$ (that is, satisfies the conditions in $\Theta_\Gamma$). 
    
    If $\Gamma\subseteq\{\msf{T}, \msf{4}, \msf{B}, \msf{5}\}$, then $\Theta_\Gamma$ is a set of universal Horn formulas, which implies that $\msf{C}$ is closed under products and substructures. As the pullback is obtained using those two constructions, we know that $\mrm{Pb}(f_1, f_2)\in\msf{C}$.

    {\highlight 
    The case where $\msf{.2}\in\Gamma$ requires more work. If $\msf{.2}\in\Gamma$, we know that $X, Y_1, Y_2$ are directed, and we show that $\mrm{Pb}(f_1, f_2)$ is directed. To do this, we first introduce a notion of \emph{weak directedness}, and prove it to be equivalent to directness for modal L-spaces.

    \begin{definition}
        A modal L-frame $(X, R)$ is \emph{weakly directed} if for all $x, y, z\in X$ with $x\mrel{R}y$, $x\mrel{R}z$ and $y\preceq z$, there exists $t\in X$ such that $y\mrel{R}t$ and $z\mrel{R}t$.
    \end{definition}

    \begin{claim}
        If a modal L-space is weakly directed, then it is directed.
    \end{claim}
    \begin{proof}
        Let $(X, \tau, R)$ be a modal space and take $x, y, z\in X$ such that $x\mrel{R}y$ and $x\mrel{R}z$. By Definition~\ref{def:modal_L-frame}\ref{item:def:modal_L-frame:iv}, we have $x\mrel{R}y\curlywedge z$, thus by weak directedness, there are $u, v\in X$ such that $y\mrel{R}u$, $y\curlywedge z\mrel{R}u$, $z\mrel{R}v$ and $y\curlywedge z\mrel{R}v$. Since $y\curlywedge z\mrel{R}u$ and $y\curlywedge z\preceq z$, by Definition~\ref{def:modal_L-frame}\ref{item:def:modal_L-frame:ii}, we can find $t\in X$ such that $u\preceq t$ and $z\mrel{R}t$. Similarly, we can find $w\in X$ such that $v\preceq w$ and $y\mrel{R}w$. Then, observe that $y\mrel{R}u\curlywedge w$, as $y\mrel{R}u$ and $y\mrel{R}w$. As $u\curlywedge w\preceq t\curlywedge w\preceq w$, and the set $R[y]$ is convex and contains $u\curlywedge w$ and $w$, we get $y\mrel{R}t\curlywedge w$. Similarly, we get $z\mrel{R}t\curlywedge w$.
    \end{proof}

    \begin{claim}
        If $\mrm{Pb}(f_1, f_2)$ is weakly directed, then is is directed.
    \end{claim}
    \begin{proof}
        The proof is essentially the same, as the only property of modal L-space that is used in the proof of the previous claim is the convexity of the rooted $R$-sets.
    \end{proof}

    \begin{claim}
        The frame $\mrm{Pb}(f_1, f_2)$ is weakly directed.
    \end{claim}
    \begin{proof}
        Take $(x_1, x_2), (y_1, y_2), (z_1, z_2)\in \mrm{Pb}(f_1, f_2)$ such that $(x_1, x_2)\mrel{R}(y_1, y_2)$, $(x_1, x_2)\mrel{R}(z_1, z_2)$, and $(y_1, y_2)\preceq(z_1, z_2)$. Let $x \coloneq f_1(x_1) = f_2(x_2)$, $y \coloneq f_1(y_1) = f_2(y_2)$ and $z \coloneq f_1(z_1) = f_2(z_2)$. Let $t$ be minimal in $R[z]$ by Corollary~\ref{coro:succ_has_max}.

        We have $f_1(z_1)\mrel{R}t$, thus by the bounded L-morphism condition, there is $t_1\in Y_1$ such that $z_1\mrel{R}t_1$ and $f_1(t_1)\preceq t$. By Corollary~\ref{coro:succ_has_max}, we may replace $t_1$ by a minimal element in $R[z_1]$, and this preserve the validity of the inequality $f_1(t_1)\preceq t$.

        Since $y_1\preceq z_1$ and $z_1\mrel{R}t_1$, by Definition~\ref{def:modal_L-frame}\ref{item:def:modal_L-frame:i}, we can find $u\in Y_1$ such that such that $y_1\mrel{R}y$ and $u\preceq t_1$. Since $Y_1$ is directed, we can find $v\in Y_1$ such that $y_1\mrel{R}z_1$.

        As $z_1\mrel{R}t_1$ and $z_1\mrel{R}v$, we have $z_1\mrel{R}t_1\curlywedge v$, thus $t_1\preceq v$, by minimality of $t_1$. Therefore, $u\preceq t_1\preceq v$, and as $u, v\in R[y_1]$, we get $t_1\in R[z_1]$, by convexity of $z_1$. We have $f_1(t_1)\preceq t$, which implies $f_1(t_1)=t$ by minimality of $t$. Therefore, we have found $t_1\in Y_1$ such that $y_1\mrel{R}t_1$, $z_1\mrel{R}t_1$ and $f_1(t_1)=t$. Similarly, we find $t_2\in Y_2$ such that $y_2\mrel{R}t_2$, $z_2\mrel{R}t_2$ and $f_2(t_2)=t$. The pair $(t_1, t_2)$ then belongs to $\mrm{Pb}(f_1, f_2)$ and witnesses weak directedness, completing the proof.
    \end{proof}

    By directedness, $R[y]\cap R[z]$ is nonempty. It is a closed set, therefore we can pick a minimal point $t\in R[y]\cap R[z]$.

    We have $f_1(y_1)\mrel{R}t$, thus by the bounded L-morphism condition, there is $t'_1\in Y_1$ such that $y_1\mrel{R}t'_1$ and $f_1(t_1)\preceq t$. As $f_1(z_1)\mrel{R}t$, we also find $t''_1\in Y_1$ such that $z_1\mrel{R}t''_1$ and $f_1(t''_1)\preceq t$. As $Y_1$ is directed, we find $t'''_1\in R[y_1]\cap R[z_1]$. Let $t_1 \coloneq (t'_1\curlywedge t'''_1)\curlyvee(t''_1\curlywedge t'''_1)$. 
        
    We have $f_1(t_1)\preceq t$. Indeed, this follows from
    \[ f_1(t'_1\curlywedge t'''_1)\preceq f_1(t'_1)\preceq t \qquad\text{and}\qquad f_1(t''_1\curlywedge t'''_1)\preceq f_1(t''_1)\preceq t. \]
    We have $y_1\mrel{R}t_1$. Indeed, $t'_1\curlywedge t'''_1\preceq t_1\preceq t'''_1$, and since $t'_1, t'''_1\in R[y_1]$, we obtain $t'_1\curlywedge t'''_1\in R[y_1]$. Then $t_1\in R[y_1]$ by convexity. Similarly, we have $z_1\mrel{R}t_1$.

    From this, it follows that $f_1(t_1)\in R[y]\cap R[z]$. As $t$ is minimal in that set, from $t\preceq f_1(t_1)$ we obtain $t = f_1(t_1)$. Proceeding similarly, we find $t_2\in Y_2$ such that $t_2\in R[y_2]\cap R[z_2]$ and $f_2(t_2) = t$. The pair $(t_1, t_2)$ then belongs to $\mrm{Pb}(f_1, f_2)$ and is such that $(y_1, y_2)\mrel{R}(t_1, t_2)$ and $(z_1, z_2)\mrel{R}(t_1, t_2)$, showing that $\mrm{Pb}(f_1, f_2)$ is a directed L-frame.}
\end{proof}

\begin{coro}
    Let $\Gamma\subseteq\{\msf{T}, \msf{4}, \msf{B}, \msf{5}, \msf{.2}\}$. Then $\mcal{L}(\Gamma)$ has the Craig interpolation property.
\end{coro}
\begin{rmk}
    One should observe that in the classical case, the directed modal logic $\msf{K.2}$ does not have Craig interpolation \cite[Section 5.6.2]{marx1995}. This makes the above result quite surprising. This can be explained by the fact that when working with modal L-spaces, many quantifiers can be eliminated by picking suitable maximal or minimal elements: in short, the usual impediments to amalgamation deriving from the need for choices disappear, since these choices become canonical.
\end{rmk}

\simnote{I still think that $\msf{K.3}$ should lack Craig interpolation (as the problem arises from an disjunction), I want to think of a counterexample. Problem is, they get real big.}

\begin{rmk}
    The above results point to some general connections between classical modal logic, positive modal logic and weak positive modal logic. For instance, if $\Gamma$ is a set of positive formulas, and $\mcal{L}(\Gamma)$, the smallest normal modal logic containing $\Gamma$, has the following properties:
    \begin{itemize}
        \item $\mcal{L}(\Gamma)$ has the \emph{Lyndon interpolation property}, and
        \item $\mcal{L}(\Gamma)$ is conservative over its positive fragment,
    \end{itemize}
    then the positive fragment of $\mcal{L}(\Gamma)$ will have the Craig interpolation property.   It would be interesting to investigate this phenomenon, as well as the relationship between positive modal logics and weak positive modal logics which are their fragments.
\end{rmk}

\begin{rmk}\label{rmk:box_ext}\highlight 
    As was pointed out to us by one of the referees the results of this section also generalize to $\Box$-only expansion of lattices. For example, if in the setting of the results of Section~\ref{section:supamal_modal} we consider $\Diamond$ to be the identity, then the variety of lattices with multiplicative (resp.~contractive and idempotent) $\Box$ can be proved, using the arguments similar to the ones in this section, to have superamalgamation. 
    
    There is however, also a more elegant way to see that these results hold. Indeed, a $\Box$-operator on a lattice $L$ is essentially a meet semilattice automorphism $\Box\colon L\to L$, and HMS duality tells us that those correspond to continuous meet semilattice automorphism on the dual space. Therefore, a lattice $L$ with a normal $\Box$-operator correspond to an L-space $\msf{Fil_L}(L)$ with a normal $\boxbox$-operator $\boxbox\colon\msf{Fil_L}\to\msf{Fil_L}$, defined by $\boxbox\coloneq\Box^{-1}$. One easily verifies that a lattice homomorphism $h$ preserve $\Box$ iff the corresponding L-space morphism $\msf{Fil_L}(h)$ preserves $\boxbox$. Thus the duality from \cite{bez2024} can be extended to a duality between lattices with a normal $\Box$-operator and $\Box$-preserving homomorphisms, and L-spaces with a normal $\boxbox$-operator and $\boxbox$-preserving L-space morphisms. The results presented in Section~\ref{sec:supamal_nonmodal} can then be easily extended to show that lattices with a normal $\Box$-operator have superamalgamation. Furthermore, the same technique also applies when the $\Box$-operator is additionally required to be contractive ($\Box a\le a$) or idempotent ($\Box\Box a = \Box a$).
\end{rmk}

%\begin{rmk}\label{rmk:modal_semi}\highlight 
%    The axioms for the operators of a modal lattice can be given without any reference to the join operator. Indeed, the axiom $\Diamond (a\vee b) \leq \Diamond (a \vee \Diamond b)$ corresponds to $\Diamond a\wedge \Diamond b \leq \Diamond (a\wedge b)$. Therefore, one can introduce and study ``modal semilattices''. The corresponding logics are likely to be rather weak, but may  still be quite interesting. In any case, such ``modal semilattices'' would be interesting to study from an algebraic perspective. 
%\end{rmk}

\section{Conclusions and future work}

In this paper we proved that the variety   of modal lattices and its ``neighbours'' have the superamalgamation property via a duality between modal lattices and modal L-spaces. Consequently, the corresponding systems of modal weak positive logic have the Craig interpolation property. This includes the modal weak positive logic axiomatized by the Church--Rosser axiom, whose classical analogue is known not to have the Craig interpolation property.
Below we list several directions for future work.

The proofs given in this paper are closely related to proofs of Craig interpolation via bisimulation products \cite{bez2026, ven2007}. We leave it open whether similar proofs can be developed for the results presented here. Bisimulations for L-spaces have been studied in \cite{deg2024}. However, a general theory of bisimulations for modal L-spaces remains to be developed.

An important strengthening of the Craig interpolation property is uniform interpolation \cite{craig2026}. It is well known that the basic modal logic $\mathsf{K}$, as well as the intuitionistic propositional calculus $\mathsf{IPC}$, enjoy the uniform interpolation property. We leave it as an open problem whether (modal) weak positive logic has the uniform interpolation property.

As pointed out in the introduction, there are only two nontrivial varieties of lattices with the amalgamation property, and only the variety of all lattices has the superamalgamation property. We leave open the problem of obtaining a full characterization of varieties of modal lattices with the (super)amalgamation property---i.e., is it possible to  obtain a Maksimova style characterization of such varieties, possibly over a natural subvariety, such as over the transitivity axiom (see Theorem~\ref{thm:supamal_ext})? The duality techniques developed in this paper may be crucial for achieving this goal, as Maksimova’s characterization is also based on duality theoretic methods \cite{cz1997, mak1977}. An important component of such an approach would be the use of finite meet semilattices as dual spaces in order to construct varieties that lack the superamalgamation property.

The question of amalgamation, and more generally of Craig interpolation, in modal distributive logics also deserves further investigation. Although the variety of distributive lattices lacks the superamalgamation property, we have shown in this paper that positive (distributive) logic nonetheless has the Craig interpolation property. We leave a systematic study of Craig interpolation in distributive modal logics (a.k.a. positive modal logics) as an interesting open problem.

   {\highlight Finally, as it was noted by one of the referees, the axioms for the operators of a modal lattice can be given without any reference to the join operator. Indeed, the axiom $\Diamond (a\vee b) \leq \Diamond (a \vee \Diamond b)$ corresponds to $\Diamond a\wedge \Diamond b \leq \Diamond (a\wedge b)$. Therefore, one can introduce and study ``modal semilattices''. The corresponding logics are likely to be rather weak, but may  still be quite interesting. In any case, such ``modal semilattices'' would be interesting to study from an algebraic perspective. }

\paragraph{AI usage declaration}
{\highlight No LLM tools, or other automated/computational tools of similar nature, were used at any stage during the preparation, writing or proofreading of this article.}

{\paragraph{Acknowledgment}\highlight We are very grateful to the referees for many useful comments and corrections that substantially improved the paper. In particular, one of the referees noticed an error in the original proof of Theorem~\ref{thm:supamal_ext} and also suggested how to correct it. Remark~\ref{rmk:box_ext} has  also been inspired by the feedback of one of the referees. We also thank  Jim de Groot for stimulating discussions and many helpful pointers to the literature. }

%\begin{itemize}
%   \item Universal Horn business.
%    \item Uniform interpolation for modal lattices.
%    \item Proofs of amalgamation via bisimulations and bisimulation products.
%    \item Characterization of all modal lattices with superamalgamation / interpolation. 
%    \item Refuting superamalgamation / interpolation.
%    \item What about interpolation / superamalgamation of distributive lattices.
%\end{itemize}